\documentclass[generic,11pt,preprint]{imsart}
\RequirePackage[OT1]{fontenc}
\RequirePackage[numbers]{natbib}
\RequirePackage{hypernat}
\usepackage{amsmath,amssymb,amsthm}  
\usepackage{mathscinet}
\usepackage[utf8]{inputenc} % Any characters can be typed directly from the keyboard, eg éçñ
\usepackage{textcomp} % provide lots of new symbols
\usepackage{graphicx}  % Add graphics capabilities
\usepackage{flafter}  % Don't place floats before their definition
\usepackage{natbib} % use author/date bibliographic citations
\usepackage{bm}  % Define \bm{} to use bold math fonts
\usepackage[pdftex,bookmarks,colorlinks,breaklinks]{hyperref}  % PDF hyperlinks, with coloured links
\usepackage{url}
\usepackage{subfigure}
\usepackage{verbatim}
\usepackage{wrapfig}
\usepackage{enumerate}
\hypersetup{linkcolor=red,citecolor=blue,filecolor=cyan,urlcolor=blue} % coloured links

\def\m{\mathcal}
\def\mb{\mathbb}
\def\dist{{\rm dist}}
\def\argmin{{\rm argmin~}}
\def\sign{{\rm sign~}}

\def\act{{\rm act}}
\def\supp{{\rm supp}}
\def\eps{\varepsilon}

\startlocaldefs
\numberwithin{equation}{section}
\theoremstyle{plain}
\newtheorem{definition}{Definition}[section]
\newtheorem{theorem}{Theorem}[section]
\newtheorem{corollary}{Corollary}[section]
\newtheorem{proposition}{Proposition}[section]

\newtheorem{assumption}{Assumption}

\endlocaldefs

\begin{document}

\begin{frontmatter}
\title{Plug-in Approach to Active Learning}
\runtitle{Plug-in Approach}
%\thankstext{T1}{Partially supported by ARC Fellowship}

\begin{aug}
\author{\fnms{Stanislav} \snm{Minsker}\thanksref{t1,t2}\ead[label=e1]{sminsker@math.gatech.edu}}
\thankstext{t1}{Partially supported by ARC Fellowship, NSF Grants DMS-0906880 and CCF-0808863}
\thankstext{t2}{Mailing address:
686 Cherry street,
School of Mathematics,
Atlanta, GA 30332-0160}
%\thankstext{t3}{Second supporter of the project}
\runauthor{S. Minsker}

\affiliation{Georgia Institute of Technology}

%\address{Stanislav Minsker\\
%686 Cherry street,\\
%School of Mathematics,\\
%Atlanta, GA 30332-0160
%}
\printead{e1}\\
%\phantom{E-mail:\ }\printead*{e1}
\end{aug}

\maketitle

\begin{abstract}
We present a new active learning algorithm based on nonparametric estimators of the regression function.
Our investigation provides probabilistic bounds for the rates of convergence of the generalization error achievable by proposed method over a broad class of underlying distributions.  
We also prove minimax lower bounds which show that the obtained rates are almost tight.
%(up to logarithmic terms) by over a broad class of underlying distributions.  
\end{abstract}

\begin{keyword}
Active learning, selective sampling, model selection, classification, confidence bands
\end{keyword}

\end{frontmatter}

\section{Introduction}
Let $(S,\m B)$ be a measurable space and
let $(X, Y)\in S\times\left\{-1,1\right\}$ be a random couple with unknown distribution $P$. 
The marginal distribution of the design variable $X$ will be denoted by $\Pi$. 
Let $\eta(x):=\mb E(Y|X=x)$ be the regression function. The goal of {\it binary classification} is to predict label $Y$ based on the observation $X$. 
Prediction is based on a {\it classifier} - a measurable function $f:S\mapsto\left\{-1,1\right\}$. 
The quality of a classifier is measured in terms of its generalization error, $R(f)=\Pr\left(Y\ne f(X)\right)$. 
In practice, the distribution $P$ remains unknown but the learning algorithm has access to the {\it training data} - the i.i.d. sample
$(X_i,Y_i), \ i=1\ldots n$ from $P$.
It often happens that the cost of obtaining the training data is associated with labeling the observations $X_i$ while the pool of observations itself is almost unlimited. 
This suggests to measure the performance of a learning algorithm in terms of its {\it label complexity}, the number of labels $Y_i$ required to obtain a classifier with the desired accuracy. 
{\it Active learning} theory is mainly devoted to design and analysis of the algorithms that can take advantage of this modified framework. 
Most of these procedures can be characterized by the following property: at each step $k$, observation $X_k$ is sampled from a distribution $\hat\Pi_k$ that depends on previously obtained
$(X_i,Y_i), \ i\leq k-1$(while passive learners obtain all available training data at the same time). 
$\hat\Pi_k$ is designed to be supported on a set where classification is difficult and requires more labeled data to be collected. 
The situation when active learners outperform passive algorithms might occur when the so-called {\it Tsybakov's low noise assumption} is satisfied: there exist constants $B,\gamma>0$ such that
\begin{equation}\label{noise1}
\forall \ t>0, \ \Pi(x: |\eta(x)|\leq t)\leq Bt^{\gamma}
\end{equation}
This assumption provides a convenient way to characterize the noise level of the problem and will play a crucial role in our investigation.\\
%The set $\left\{x: |\eta(x)|\leq t\right\}$ appearing in (\ref{noise1}) can be thought of as a certain approximation of the {\it decision boundary} $\left\{x\in S:\eta(x)=0\right\}$, which is exactly the domain of uncertainty.}\\
The topic of active learning is widely present in the literature; see \citet{balcan1}, \citet{hanneke2}, \citet{castro1} for review. 
It was discovered that in some cases the generalization error of a resulting classifier can converge to zero exponentially fast with respect to its label complexity(while the best rate for passive learning is usually polynomial with respect to the cardinality of the training data set). 
However, available algorithms that adapt to the unknown parameters of the problem($\gamma$ in Tsybakov's low noise assumption, regularity of the decision boundary) involve empirical risk minimization with binary loss, along with other computationally hard problems, see \citet{hanneke1}, \citet{hanneke2}. 
On the other hand, the algorithms that can be effectively implemented, as in \citet{castro1}, are not adaptive. \\
The majority of the previous work in the field was done under standard complexity assumptions on the set of possible classifiers(such as polynomial growth of the covering numbers). 
\citet{castro1} derived their results under the regularity conditions on the decision boundary and the noise assumption which is slightly more restrictive then (\ref{noise1}). 
Essentially, they proved that if the decision boundary is a graph of the H\"{o}lder smooth function $g\in \Sigma(\beta,K,[0,1]^{d-1})$
(see section \ref{preliminaries} for definitions) and 
the noise assumption is satisfied with $\gamma>0$,
then the minimax lower bound for the expected excess risk of the active classifier is of order 
$
C\cdot N^{-\frac{\beta(1+\gamma)}{2\beta+\gamma(d-1)}}
$
and the upper bound is 
$
C (N/\log N)^{-\frac{\beta(1+\gamma)}{2\beta+\gamma(d-1)}},
$
where $N$ is the label budget. However, the construction of the classifier that achieves an upper bound assumes  $\beta$ and $\gamma$ to be known. \\
In this paper, we consider the problem of active learning under classical nonparametric assumptions on the regression function - namely, we assume that it belongs to a certain H\"{o}lder class $\Sigma(\beta,K,[0,1]^d)$ and satisfies to the low noise condition (\ref{noise1}) with some positive $\gamma$. 
In this case, the work of \citet{tsyb2} showed that plug-in classifiers can attain optimal rates in the {\it passive} learning framework, namely, that the expected excess risk of a classifier 
$\hat g=\sign \hat\eta$ 
is bounded above by 
$CN^{-\frac{\beta(1+\gamma)}{2\beta+d}}$
(which is the optimal rate), 
where $\hat \eta$ is the local polynomial estimator of the regression function and $N$ is the size of the training data set.  
We were able to partially extend this claim to the case of active learning: 
first, we obtain minimax lower bounds for the excess risk of an active classifier in terms of its label complexity.
Second, we propose a new algorithm that is based on plug-in classifiers, attains almost optimal rates over a broad class of distributions and possesses adaptivity with respect to $\beta,\gamma$(within the certain range of these parameters). \\
The paper is organized as follows: the next section introduces remaining notations and specifies the main assumptions made throughout the paper.
This is followed by a qualitative description of our learning algorithm. 
The second part of the work contains the statements and proofs of our main results - minimax upper and lower bounds for the excess risk.   

%%%%%%%Preliminaries%%%%%%%
%#############################% 
\section{Preliminaries}\label{preliminaries}
Our {\it active learning} framework is governed by the following rules: 
\begin{enumerate}
\item Observations are sampled sequentially: $X_k$ is sampled from the modified distribution $\hat \Pi_k$ that depends on 
$(X_1,Y_1),\ldots,(X_{k-1},Y_{k-1})$.
\item $Y_k$ is sampled from the conditional distribution 
$P_{Y|X}(\cdot|X=x)$. 
Labels are conditionally independent given the feature vectors $X_i, \ i\leq n$.
\end{enumerate}
Usually, the distribution $\hat \Pi_k$ is supported on a set where classification is difficult.\\
Given the probability measure $\mb Q$ on $S\times\left\{-1,1\right\}$, we denote the integral with respect to this measure by $\mb Q g:=\int g d\,\mb Q$.
Let $\m F$ be a class of bounded, measurable functions. 
The risk and the excess risk of $f\in \m F$ with respect to the measure $\mb Q$ are defined by
\begin{align*}
&
R_{\mb Q}(f):=\mb Q \m I_{y\ne \sign f(x)} \\
&
\mathcal{E}_{\mb Q}(f):=R_{\mb Q}(f)-\inf\limits_{g\in \m F}R_{\mb Q}(g),
\end{align*}
where $\m I_{\m A}$ is the indicator of event $\m A$.
We will omit the subindex $\mb Q$ when the underlying measure is clear from the context. Recall that we denoted the distribution of $(X,Y)$ by $P$.
The minimal possible risk with respect to $P$ is  
$$
\ R^*=\inf\limits_{g:S\mapsto[-1,1]}\Pr\left(Y\ne \sign g(X)\right),
$$
where the infimum is taken over all measurable functions. 
It is well known that it is attained for any $g$ such that $\sign g(x)=\sign \eta(x)$ $\Pi$ - a.s. 
Given $g\in \m F, \ A\in \m B, \ \delta>0$, define 
$$
\m F_{\infty,A}(g;\delta):=\left\{f\in \m F: \ \|f-g\|_{\infty,A}\leq \delta\right\},
$$ 
where $\|f-g\|_{\infty,A}=\sup\limits_{x\in A}|f(x)-g(x)|$. For $A\in \m B$, define the function class 
$$
\m F|_{A}:=\left\{f|_{A}, \ f\in \m F\right\}
$$
where $f|_{A}(x):=f(x)I_A(x)$.
From now on, we restrict our attention to the case $S=[0,1]^d$. Let $K>0$. 
\begin{definition}
\label{holder3}
We say that $g:\mb R^d\mapsto\mb R$ belongs to $\Sigma(\beta,K,[0,1]^d)$, the $(\beta,K,[0,1]^d)$ - H\"{o}lder class of functions, 
if $g$ is $\lfloor\beta\rfloor$ times continuously differentiable and for all $x, x_1\in [0,1]^d$ satisfies
$$
|g(x_1)-T_x(x_1)|\leq K\|x-x_1\|_{\infty}^{\beta},
$$
where $T_x$ is the Taylor polynomial of degree $\lfloor\beta\rfloor$ of $g$ at the point $x$. 
\end{definition}
\begin{definition}
$\m P(\beta,\gamma)$ is the class of probability distributions on \\
 $\mb [0,1]^d\times\left\{-1,+1\right\}$ with the following properties:
\begin{enumerate}
 \item $\forall \ t>0, \ \Pi(x: |\eta(x)|\leq t)\leq Bt^{\gamma}$;
\item $\eta(x)\in \Sigma(\beta,K,\mb [0,1]^d)$.
%for some $K>0$.
\end{enumerate}
\end{definition}
We do not mention the dependence of $\m P(\beta,\gamma)$ on the fixed constants $B,K$ explicitly, but this should not cause any uncertainty.\\ 
Finally, let us define $\m P_U^*(\beta,\gamma)$ and  $\m P_U(\beta,\gamma)$, the subclasses of $\m P(\beta,\gamma)$, 
by imposing two additional assumptions. 
Along with the formal descriptions of these assumptions,
 we shall try to provide some motivation behind them. The first deals with the marginal $\Pi$. 
For an integer $M\geq 1$, let 
$$
\m G_{M}:=\left\{\left(\frac{k_1}{M},\ldots,\frac{k_d}{M}\right), \ k_i=1\ldots M, \ i=1\ldots d\right\}
$$
be the regular grid on the unit cube $[0,1]^d$ with mesh size $M^{-1}$. 
It naturally defines a partition into a set of $M^d$ open cubes $R_i, \ i=1\ldots M^d$ with edges of length $M^{-1}$ and vertices in $\m G_{M}$.
Below, we consider the nested sequence of grids $\left\{\m G_{2^{m}}, \ m\geq 1\right\}$ and corresponding dyadic partitions of the unit cube. 
\begin{definition}
We will say that $\Pi$ is $(u_1,u_2)$-regular with respect to $\left\{\m G_{2^{m}}\right\}$ if for any $m\geq 1$, any element of the partition 
$R_i, \ i\leq 2^{dm}$ such that $R_i\cap {\rm supp}(\Pi)\ne \emptyset$, 
we have  
\begin{equation}\label{t-reg}
u_1\cdot 2^{-dm}\leq \Pi\left(R_i\right)\leq u_2\cdot 2^{-dm}.
\end{equation}
where $0<u_1\leq u_2<\infty$.
\end{definition}
\begin{assumption}\label{regular}
$\Pi$ is $(u_1,u_2)$ - regular.
\end{assumption}
In particular, $(u_1,u_2)$-regularity holds for the distribution with a density $p$ on $[0,1]^d$ such that $0<u_1\leq p(x)\leq u_2<\infty$.\\
Let us mention that our definition of regularity is of rather technical nature; 
for most of the paper, the reader might think of $\Pi$ as being uniform on $[0,1]^d$( 
however, we need slightly more complicated marginal to construct the minimax lower bounds for the excess risk). 
It is know that estimation of regression function in sup-norm is sensitive to the geometry of design distribution, mainly because the quality of estimation depends on the {\it local} amount of data at every point; 
conditions similar to our {\it assumption} \ref{regular} were used in the previous works where this problem appeared, e.g., {\it strong density assumption} in \citet{tsyb2} and {\it assumption D} in \citet{gaiffas1}.\\
Another useful characteristic of $(u_1,u_2)$ - regular distribution $\Pi$ is that this property is stable with respect to restrictions of $\Pi$ to certain subsets of its support. 
This fact fits the active learning framework particularly  well.
\begin{definition}
We say that $\mb Q$ belongs to $\m P_U(\beta,\gamma)$ if $\mb Q\in \m P(\beta,\gamma)$ and {\it assumption} \ref{regular} is satisfied for some $u_1, u_2$.
\end{definition}
The second assumption is crucial in derivation of the upper bounds. 
The space of piecewise-constant functions which is used to construct the estimators of $\eta(x)$ is defined via
$$
\m F_m=\left\{\sum\limits_{i=1}^{2^{dm}} \lambda_i I_{R_i}(\cdot): \ |\lambda_i|\leq 1,
\ i=1\ldots 2^{dm}\right\},
$$ 
where $\left\{R_i\right\}_{i=1}^{2^{dm}}$ forms the dyadic partition of the unit cube. 
Note that $\m F_m$ can be viewed as a $\|\cdot\|_{\infty}$-unit ball in the linear span of first $2^{dm}$ Haar basis functions in $[0,1]^d$. 
Moreover, $\left\{\m F_m, \ m\geq 1\right\}$ is a nested family, which is a desirable property for the model selection procedures.
By $\bar\eta_m(x)$ we denote the $L_2(\Pi)$ - projection of the regression function onto $\m F_m$.\\
We will say that the set $A\subset [0,1]^d$ {\it approximates the decision boundary} $\left\{x:\eta(x)=0\right\}$ if there exists $t>0$ such that 
\begin{equation}\label{approx}
\left\{x:|\eta(x)|\leq t\right\}_{\Pi}\subseteq A_{\Pi}\subseteq \left\{x:|\eta(x)|\leq3t\right\}_{\Pi},
\end{equation}
where for any set $A$ we define $A_{\Pi}:=A\cap \supp(\Pi)$.
The most important example we have in mind is the following: 
let $\hat \eta$ be some estimator of $\eta$ with
$
\|\hat \eta-\eta\|_{\infty,\supp(\Pi)}\leq t,
$
and define the $2t$ - band around $\eta$ by 
$$
\hat F=\left\{f: \  \hat \eta(x)-2t\leq f(x)\leq \hat \eta(x)+2t \ \forall x\in[0,1]^d\right\}
$$
Take
$
A=\left\{x: \ \exists f_1,f_2\in \hat F \text{ s.t. }\sign f_1(x)\ne\sign f_2(x)\right\}
$,
then it is easy to see that $A$ satisfies (\ref{approx}). 
Modified design distributions used by our algorithm are supported on the sets with similar structure.\\
Let $\sigma(\m F_m)$ be the sigma-algebra generated by $\m F_m$ and $A\in \sigma(\m F_m)$. 
\begin{assumption} \label{holder2}
There exists $B_2>0$ such that for all $m\geq 1$, $A\in \sigma(\m F_m)$ satisfying (\ref{approx}) and such that $A_{\Pi}\ne\emptyset$ 
the following holds true:
\begin{align*}
& 
 \int\limits_{[0,1]^d}\left(\eta-\bar\eta_m\right)^2 \Pi(dx|x\in A_{\Pi})
 \geq B_2 \|\eta-\bar\eta_m\|^2_{\infty,A_{\Pi}}
\end{align*} 
\end{assumption}  
Appearance of {\it assumption \ref{holder2}} is motivated by the structure of our learning algorithm - namely, 
it is based on adaptive confidence bands for the regression function. 
Nonparametric confidence bands is a big topic in statistical literature, and the review of this subject is not our goal. 
We just mention that it is impossible to construct adaptive confidence bands of optimal size over the whole 
$\bigcup\limits_{\beta\leq 1}\Sigma\left(\beta, K,[0,1]^d\right)$. \citet{low1,nickl2} discuss the subject in details. 
However, it is possible to construct adaptive $L_2$ - confidence balls(see an example following Theorem 6.1 in \citet{kolt6}). 
For functions satisfying {\it assumption \ref{holder2}}, 
this fact allows to obtain confidence bands of desired size. 
In particular, 
\begin{enumerate}[(a)]
\item
functions that are differentiable, with gradient being bounded away from 0 in the vicinity of decision boundary;
\item Lipschitz continuous functions that are convex in the vicinity of decision boundary
\end{enumerate}
satisfy {\it assumption} \ref{holder2}.
For precise statements, see Propositions \ref{holder2:example1}, \ref{holder2:example2} in Appendix \ref{examples}. 
A different approach to adaptive confidence bands in case of one-dimensional density estimation is presented in \citet{nickl1}.
Finally, we define $\m P_U^*(\beta,\gamma)$:
\begin{definition}
We say that $\mb Q$ belongs to $\m P_U^*(\beta,\gamma)$ if $\mb Q\in \m P_U(\beta,\gamma)$ and {\it assumption} \ref{holder2} is satisfied for some $B_2>0$.
\end{definition}
%%%%%%%%%%%%%%%%%%%%%%%%%%%%%%%%%%
\subsection{Learning algorithm}
Now we give a brief description of the algorithm, since several definitions appear naturally in this context. 
First, let us emphasize that {\it the marginal distribution $\Pi$ is assumed to be known to the learner.} 
This is not a restriction, since we are not limited in the use of unlabeled data and $\Pi$ can be estimated to any desired accuracy. 
Our construction is based on so-called {\it plug-in} classifiers of the form 
$\hat f(\cdot)=\sign \hat \eta(\cdot)$, 
where $\hat\eta$ is a piecewise-constant estimator of the regression function. 
As we have already mentioned above, it was shown in \citet{tsyb2} that in the passive learning framework plug-in classifiers attain optimal rate for the excess risk of order $N^{-\frac{\beta(1+\gamma)}{2\beta+d}}$, 
with $\hat \eta$ being the local polynomial estimator.
\par
Our active learning algorithm iteratively improves the classifier by constructing shrinking confidence bands for the regression function. 
On every step $k$, the piecewise-constant estimator $\hat\eta_k$ is obtained via the model selection procedure 
which allows adaptation to the unknown smoothness(for H\"{o}lder exponent $\leq1$). 
The estimator is further used to construct a confidence band $\hat {\m F}_k$ for $\eta(x)$. 
The {\it active set} assosiated with $\hat{\m F}_k$ is defined as
$$
\hat A_k=A(\hat {\m F}_k):=\left\{x\in {\supp(\Pi)}: \ \exists f_1,f_2\in \hat {\m F}_k, \sign f_1(x)\ne\sign f_2(x)\right\}
$$ 
Clearly, this is the set where the confidence band crosses zero level and where classification is potentially difficult. 
$\hat A_k$ serves as a support of the modified distribution $\hat \Pi_{k+1}$: on step $k+1$, label $Y$ is requested only for observations $X\in \hat{A}_k$, 
forcing the labeled data to concentrate in the domain where higher precision is needed. 
This allows one to obtain a tighter confidence band for the regression function restricted to the active set. 
Since $\hat A_k$ approaches the decision boundary, its size is controlled by the low noise assumption. 
The algorithm does not require a priori knowledge of the noise and regularity parameters, 
being adaptive for $\gamma>0, \beta\leq 1$.
\begin{figure}[ht]
\centering
\label{fig}
\subfigure{\label{fig:1}\includegraphics[scale=0.31]{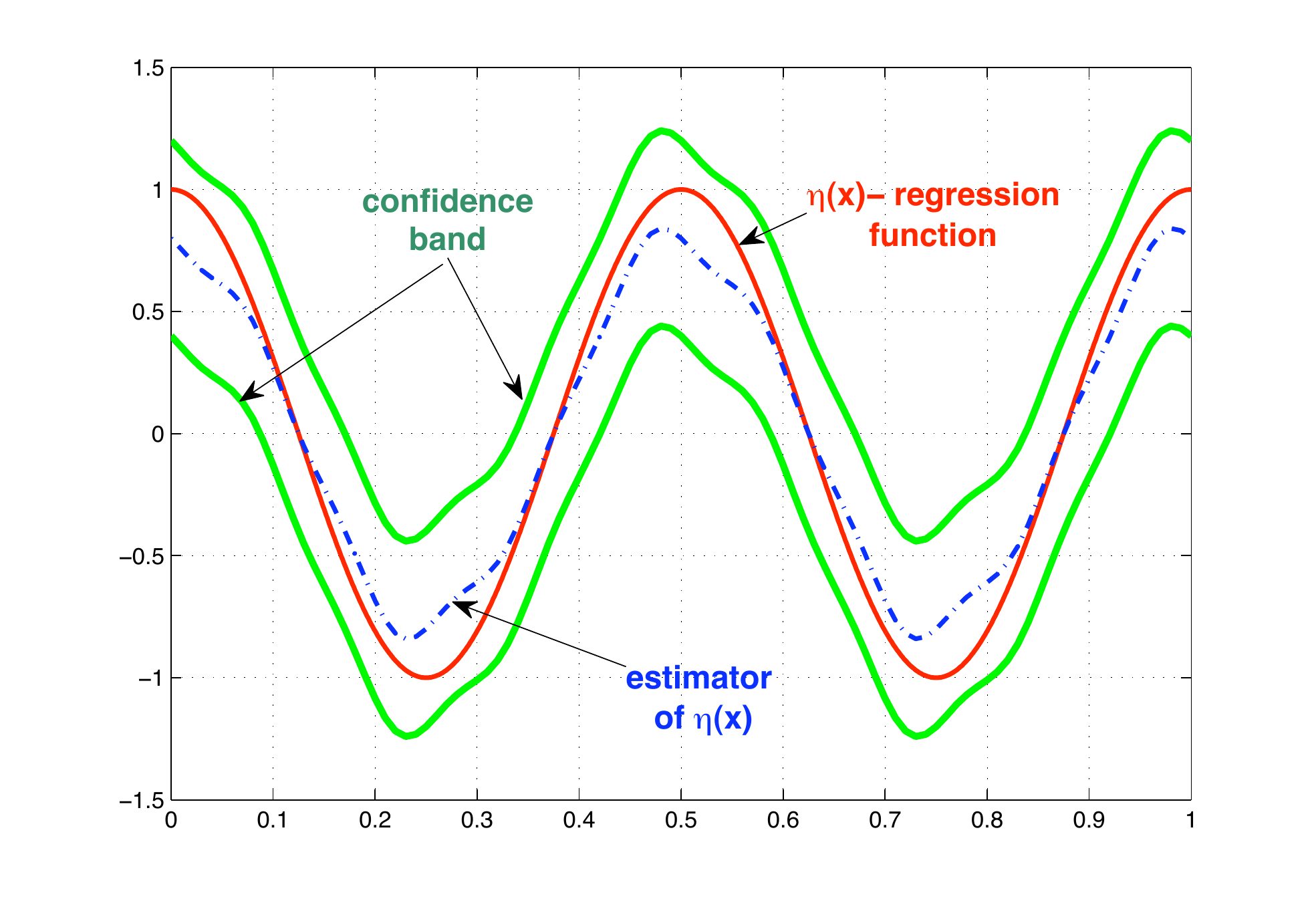}}
\subfigure{\label{fig:2}\includegraphics[scale=0.30]{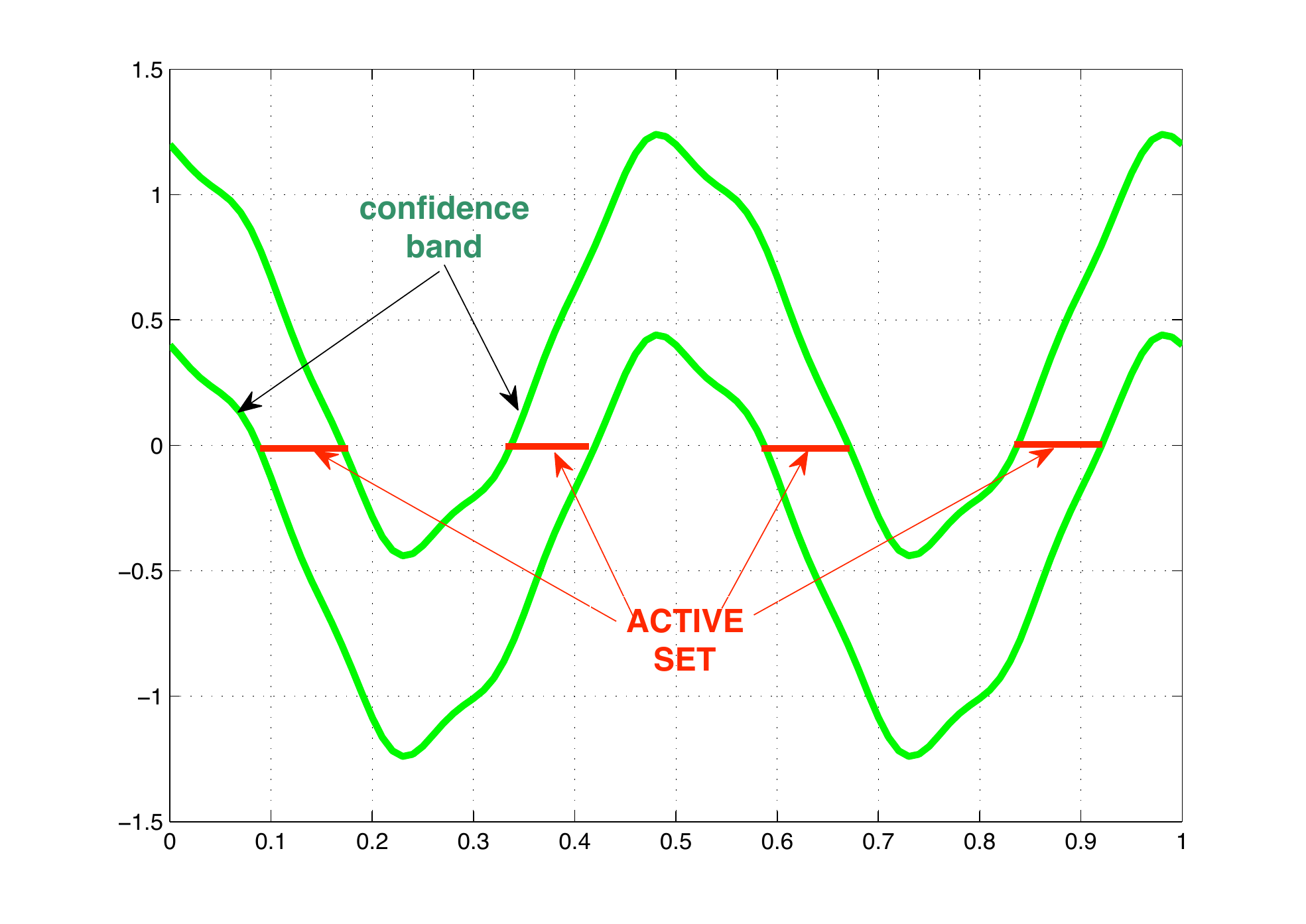}}
\caption{Active Learning Algorithm}
\end{figure}
Further details are given in section \ref{upper}.
%##########################
%\subsection{Basic facts}
%##########################
\subsection{Comparison inequalities} Before proceeding to the main results, let us recall the well-known connections between the binary risk and the $\|~\cdot~\|_{\infty}$, 
$\|\cdot~\|_{L_2(\Pi)}$ - norm risks:
\begin{proposition}\label{risk_bound}
Under the low noise assumption,
\begin{align}
&\label{sup}
R_{P}(f)-R^*\leq D_1\|(f-\eta)\m I\left\{\sign f\ne \sign \eta\right\}\|_{\infty}^{1+\gamma}; \\
&\label{square}
R_{P}(f)-R^*\leq D_2\|(f-\eta)\m I\left\{\sign f\ne \sign \eta\right\}\|_{L_2(\Pi)}^{\frac{2(1+\gamma)}{2+\gamma}}; \\
&\label{noise}
R_P(f)-R^*\geq D_3\Pi(\sign f\ne \sign \eta)^{\frac{1+\gamma}{\gamma}}
\end{align}
\end{proposition}
\begin{proof}
For (\ref{sup}) and (\ref{square}), see \citet{tsyb2}, lemmas 5.1, 5.2 respectively, and for (\ref{noise})---\citet{kolt6}, lemma 5.2.
\end{proof}

%########################################
\section{Main results}

The question we address below is: what are the best possible rates that can be achieved by active algorithms in our framework and how these rates can be attained.

\subsection{Minimax lower bounds for the excess risk}

The goal of this section is to prove that for $P\in \m P(\beta,\gamma)$ 
no active learner can output a classifier with expected excess risk converging to zero faster than $N^{-\frac{\beta(1+\gamma)}{2\beta+d-\beta\gamma}}$. 
Our result builds upon the minimax bounds of \citet{tsyb2}, \citet{castro1}. \\
{\bf Remark}
The theorem below is proved for a smaller class $\m P_U^*(\beta,\gamma)$, which implies the result for $\m P(\beta,\gamma)$.
\begin{theorem}\label{lower_bound}
Let $\beta,\gamma,d$ be such that $\beta\gamma\leq d$. Then there exists $C>0$ such that for all $n$ large enough and for any active classifier $\hat f_n(x)$ we have
$$
\sup_{P\in \m{P}_U^*(\beta,\gamma)}\mb E R_P(\hat f_n)-R^*\geq C N^{-\frac{\beta(1+\gamma)}{2\beta+d-\beta\gamma}}
$$ 
\end{theorem}
\begin{proof}
We proceed by constructing the appropriate family of classifiers $f_{\sigma}(x)=\sign \eta_{\sigma}(x)$, in a way similar to Theorem 3.5 in \citet{tsyb2}, and then apply Theorem 2.5 from \citet{tsyb3}. We present it below for reader's convenience.
\begin{theorem}
\label{tsyb2.5}
Let $\Sigma$ be a class of models, ${\rm d}: \Sigma\times \Sigma\mapsto \mb R$ - the pseudometric and $\left\{P_{f}, \ f\in \Sigma\right\}$ - a collection of probability measures associated with $\Sigma$. 
Assume there exists a subset  $\left\{f_{0},\ldots,f_{M}\right\}$ of $\Sigma$ such that 
\begin{enumerate}
\item $d(f_{i},f_{j})\geq 2s>0 \ \forall 0\leq i<j\leq M$
\item  $P_{f_{j}}\ll P_{f_{0}}$ for every $1\leq j\leq M$
\item $\frac 1 M \sum_{j=1}^{M}{\rm KL}(P_{f_{j}},P_{f_{0}})\leq \alpha \log M, \quad 0<\alpha<\frac 18$
\end{enumerate}
Then 
$$
\inf_{\hat f}\sup_{f\in \Sigma}P_{f}\left({\rm d}(\hat f,f)\geq s\right)\geq
 \frac{\sqrt{M}}{1+\sqrt{M}}\left(1-2\alpha-\sqrt{\frac{2\alpha}{\log M}}\right)
$$
where the infimum is taken over all possible estimators of $f$ based on a sample from $P_{f}$ 
and ${\rm KL}(\cdot,\cdot)$ is the Kullback-Leibler divergence.
\end{theorem} 
Going back to the proof, let $q=2^{l}, \ l\geq 1$ and
$$
G_q:=\left\{\left(\frac{2k_1-1}{2q},\ldots,\frac{2k_d-1}{2q}\right), \ k_i=1\ldots q, \ i=1\ldots d\right\}
$$
be the grid on $[0,1]^d$. For $x\in [0,1]^d$, let 
\begin{align*}
n_q(x)&=\argmin\left\{\|x-x_k\|_2: \ x_k\in G_q\right\}
%n_q(x)&=\argmin\left\{\|y\|_2: y\in \hat{n}_q(x)\right\}
\end{align*}
If $n_q(x)$ is not unique, we choose the one with smallest $\|\cdot\|_2$ norm.
The unit cube is partitioned with respect to $G_q$ as follows: $x_1,x_2$ belong to the same subset if
$n_q(x_1)=n_q(x_2)$. Let $'\succ'$ be some order on the elements of $G_q$ such that $x\succ y$ implies $\|x\|_{2}\geq \|y\|_{2}$. 
Assume that the elements of the partition are enumerated with respect to the order of their centers induced by $'\succ'$: $[0,1]^d=\bigcup\limits_{i=1}^{q^d}R_i$.
Fix $1\leq m\leq q^d$ and let 
\begin{align*}
&
S:=\bigcup_{i=1}^m R_i
%&
%S_{1/q}:=\left\{x\in [0,1]^d: \dist_{\infty}(x,S)\leq \frac 1q\right\},
\end{align*}
%where $\dist_{\infty}(x,A):=\inf\left\{\|x-z\|_{\infty}: \ z\in A\right\}$. 
Note that the partition is ordered in such a way that there always exists $1\leq k\leq q\sqrt{d}$ with 
\begin{equation}
\label{cube}
B_+\left(0,\frac{k}{q}\right)\subseteq S\subseteq B_+\left(0,\frac{k+3\sqrt{d}}{q}\right),
\end{equation}
where $B_+(0,R):=\left\{x\in \mb R^d_+: \ \|x\|_2\leq R \right\}$. 
In other words, (\ref{cube}) means that that the difference between the radii of inscribed and circumscribed spherical sectors of $S$ is of order $C(d) q^{-1}$.\\
Let $v>r_1>r_2$ be three integers satisfying 
\begin{equation}
\label{constants}
2^{-v}<2^{-r_1}<2^{-r_1}\sqrt{d}<2^{-r_2}\sqrt{d}<2^{-1}
\end{equation}
Define $u(x):\mb R\mapsto \mb R_+$ by
\begin{equation}
\label{smooth}
u(x):=\frac{\int_x^{\infty}U(t)dt}{\int\limits_{2^{-v}}^{1/2} U(t)dt} 
\end{equation}
where
$$
U(t):=\left\{\begin{array}{c l}
\exp\left(-\frac{1}{(1/2-x)(x-2^{-v})}\right), & x\in(2^{-v},\frac12)\\
0 & \text{else.} \\
\end{array}\right.
$$
Note that $u(x)$ is an infinitely diffferentiable function such that 
$u(x)=1, \ x\in [0,2^{-v}]$ and $u(x)=0, \ x\geq \frac12$. 
Finally, for 
$x\in \mb R^d$ let 
$$
\Phi(x):=C u(\|x\|_{2})
$$ 
where $C:=C_{L,\beta}$ is chosen such that $\Phi\in \Sigma(\beta, L, \mb R^d)$.  \\ 
Let 
$
r_S:=\inf\left\{r>0: \ B_+(0,r)\supseteq S\right\}
$ and
$$
A_0:=\left\{\bigcup\limits_{i}R_i: \ R_i\cap B_+\left(0,r_S+q^{-\frac{\beta\gamma}{d}}\right)=\emptyset\right\}
$$
% $A_0:=\left\{\bigcup\limits_{i}R_i: \ R_i\cap B_+(0,r_S+\frac{\sqrt{d}}{q})=\emptyset\right\}$.
\begin{figure}[t]
\label{minmax}
\begin{center}
\includegraphics[width=0.33\textwidth]{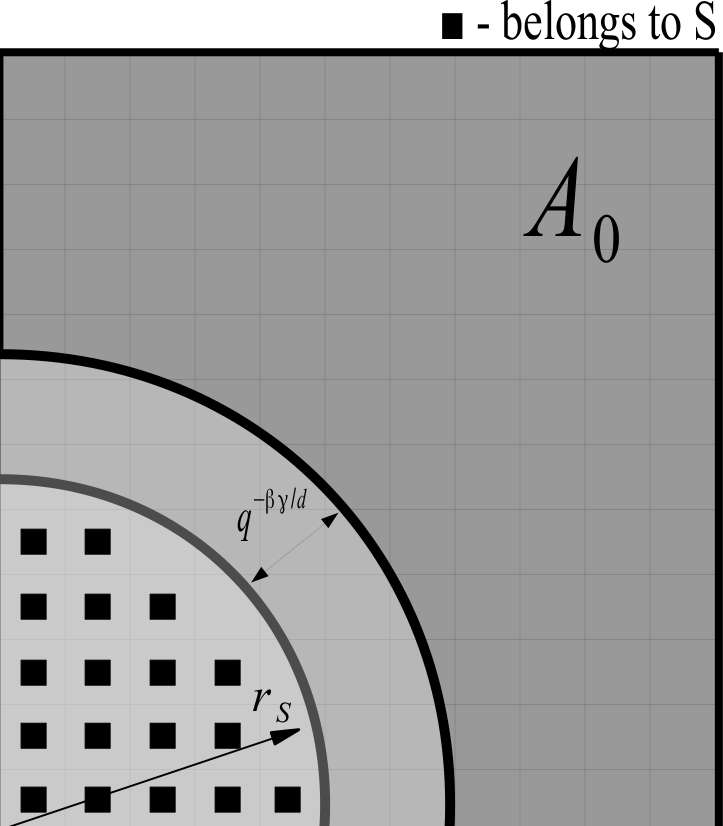}
\end{center}
\caption{Geometry of the support}
\end{figure}
Note that 
\begin{equation}
\label{radius}
r_S\leq c\frac{m^{1/d}}{q},
\end{equation} 
since ${\rm Vol}\, (S)=mq^{-d}$.
\\
Define 
$
\m H_m=\left\{P_{\sigma}: \sigma\in\left\{-1,1\right\}^m\right\}
$
to be the hypercube of probability distributions on 
$[0,1]^d\times \left\{-1,+1\right\}$. 
The marginal distribution $\Pi$ of $X$ is independent of $\sigma$: define its density $p$ by
$$
p(x)=\left\{\begin{array}{c l}
\frac{2^{d(r_1-1)}}{2^{d(r_1-r_2)}-1}, & \ x\in B_{\infty}\left(z,\frac{2^{-r_2}}{q}\right)\setminus B_{\infty}\left(z,\frac{2^{-r_1}}{q}\right), \ z\in G_q\cap S, \\
c_0, & x\in A_0, \\
0 & \text{else}.             
            \end{array}\right.
$$
where $B_{\infty}(z,r):=\left\{x: \ \|x-z\|_{\infty}\leq r\right\}$, $c_0:=\frac{1-mq^{-d}}{{\rm Vol}(A_0)}$(note that $\Pi(R_i)=q^{-d}\quad\forall i\leq m$) and $r_1, r_2$ are defined in (\ref{constants}). 
In particular, $\Pi$ satisfies {\it assumption} {\ref{regular}} since it is supported on the union of dyadic cubes and has bounded above and below on $\supp(\Pi)$ density.

Let 
$$
\Psi(x):=u\left(1/2-q^{\frac{\beta\gamma}{d}}\dist_2(x,B_+(0,r_S))\right),
$$
where $u(\cdot)$ is defined in (\ref{smooth}) and 
$\dist_2(x,A):=\inf\left\{\|x-y\|_2, \ y\in A\right\}$.\\
Finally, the regression function  
$\eta_{\sigma}(x)=\mb E_{P_{\sigma}}(Y|X=x)$ 
is defined via
$$
\eta_{\sigma}(x):=
\left\{\begin{array}{l l}
\sigma_i q^{-\beta}\Phi(q[x-n_q(x)]), & x\in R_i, \ 1\leq i\leq m \\
\frac{1}{C_{L,\beta}\sqrt{d}}\,
\dist_2(x,B_+(0,r_S))^{\frac{d}{\gamma}}\cdot \Psi(x),
 & x\in [0,1]^d\setminus S.               
\end{array}
\right.
$$
The graph of $\eta_{\sigma}$ is a surface consisting of small ''bumps" spread around $S$ 
and tending away from 0 monotonically with respect to 
$\dist_{2}(\cdot,B_+(0,r_S))$ on $[0,1]^d\setminus S$. 
Clearly, $\eta_{\sigma}(x)$ satisfies smoothness requirement, since for $x\in[0,1]^d$
$$
\dist_2(x,B_+(0,r_S))=\|x\|_2-r_S
$$ 
and $\frac{d}{\gamma}\geq \beta$ by assumption.
\footnote{$\Psi(x)$ can be replaced by 1 unless $\beta\gamma=d$ and $\beta$ is an integer, in which case extra smoothness at the boundary of $B_+(0,r_S)$, provided by $\Psi$, is necessary.}
Let's check that it also satisfies the low noise condition.
Since $|\eta_{\sigma}|\geq Cq^{-\beta}$ on support of $\Pi$, it is enough to consider 
$t=Czq^{-\beta}$ for $z>1$:
%\begin{enumerate}[(a)]
%\item $C_1 q^{-\frac{d}{\gamma}}\leq t \leq C_2 q^{-\beta}$:
%$$
%\Pi(|\eta_{\sigma}(x)|\leq t)\leq \Pi(\dist_2(x,B_+(0,r_S))\leq C t^{\gamma/d})
%$$
%\end{enumerate}
\begin{align*}
\Pi(|\eta_{\sigma}(x)|\leq Czq^{-\beta})&\leq 
mq^{-d}+\Pi\left(\dist_{2}(x,B_+(0,r_S))\leq Cz^{\gamma/d}q^{-\frac{\beta\gamma}{d}}\right)\leq \\
&
\leq mq^{-d}+C_2\left(r_S+Cz^{\gamma/d}q^{-\frac{\beta\gamma}{d}}\right)^d\leq  \\
&
\leq mq^{-d}+C_3 mq^{-d}+C_4 z^{\gamma}q^{-\beta\gamma}\leq \\
&
\leq \widehat C t^{\gamma},
\end{align*}
if $mq^{-d}=O(q^{-\beta\gamma})$. 
Here, the first inequality follows from considering $\eta_{\sigma}$ on $S$ and $A_0$ 
separately, 
and second inequality follows from (\ref{radius}) and direct computation of the sphere volume.
\\
Finally, $\eta_{\sigma}$ satisfies {\it assumption} \ref{holder2} with some $B_2:=B_2(q)$ since on $\supp(\Pi)$
$$
0<c_1(q)\leq \|\nabla \eta_{\sigma}(x)\|_2\leq c_2(q)<\infty
$$ 
The next step in the proof is to choose the subset of $\m H$ which is ``well-separated'': this can be done due to the following fact(see \citet{tsyb3}, Lemma 2.9): 
\begin{proposition}[Gilbert-Varshamov]
For $m\geq 8$, there exists 
$$
\left\{\sigma_0,\ldots,\sigma_M\right\}\subset \left\{-1,1\right\}^m
$$
such that 
$\sigma_0=\left\{1,1,\ldots,1\right\}$, $\rho(\sigma_i,\sigma_j)\geq \frac{m}{8} \ \forall  \ 0\leq i<k\leq M$ and
$M\geq 2^{m/8}$
% and $\sigma_j$ is typical for all $0\leq j \leq M$,
where $\rho$ stands for the Hamming distance. 
\end{proposition}
Let 
$\m H':=\left\{P_{\sigma_0},\ldots, P_{\sigma_M}\right\}$ be chosen such that $\left\{\sigma_0,\ldots,\sigma_M\right\}$ satisfies the proposition above.
%where $\rho$ stands for the Hamming distance. 
Next, following the proof of Theorems 1 and 3 in \citet{castro1}, 
we note that $\forall \sigma\in \m H', \ \sigma\ne\sigma_0$
\begin{align}
\label{KL}
&
{\rm KL}(P_{\sigma,N}\|P_{\sigma_0,N})\leq 8N\max_{x\in[0,1]}(\eta_{\sigma}(x)-\eta_{\sigma_0}(x))^2\leq 32 C_{L,\beta}^2 N q^{-2\beta},
\end{align}
where $P_{\sigma,N}$ is the joint distribution of $\left(X_i,Y_i\right)_{i=1}^N$ under hypothesis that the distribution of couple $(X,Y)$ is $P_{\sigma}$.
Let us briefly sketch the derivation of (\ref{KL});
see also the proof of Theorem 1 in \citet{castro1}. 
Denote 
\begin{align*}
&
\bar X_k:=(X_1,\ldots,X_k), \\
& 
\bar Y_k\,:=(Y_1,\ldots,Y_k)
\end{align*}
Then $d P_{\sigma,N}$ admits the following factorization:
\begin{align*}
&
d P_{\sigma,N}(\bar X_N,\bar Y_N)=\prod_{i=1}^N P_{\sigma}(Y_i|X_i)dP(X_i|\bar X_{i-1},\bar Y_{i-1}),
\end{align*}
where $dP(X_i|\bar X_{i-1},\bar Y_{i-1})$ does not depend on $\sigma$ but only on the active learning algorithm. As a consequence,
\begin{align*}
{\rm KL}(P_{\sigma,N}\|P_{\sigma_0,N})&=
\mb E_{P_{\sigma,N}}\log\frac{dP_{\sigma,N}(\bar X_N,\bar Y_N)}{dP_{\sigma_0,N}(\bar X_n,\bar Y_N)}=
\mb E_{P_{\sigma,N}}\log\frac{\prod_{i=1}^N P_{\sigma}(Y_i|X_i)}{\prod_{i=1}^N P_{\sigma_0}(Y_i|X_i)}=\\
&
=\sum_{i=1}^N \mb E_{P_{\sigma,N}} \left[\mb E_{P_{\sigma}}\left(\log\frac{P_{\sigma}(Y_i|X_i)}{P_{\sigma_0}(Y_i|X_i)}\left\lvert X_i\right.\right)\right]\leq \\
&
\leq N\max_{x\in[0,1]^d}\mb E_{P_{\sigma}}\left(\log\frac{P_{\sigma}(Y_1|X_1)}{P_{\sigma_0}(Y_1|X_1)}\left\lvert X_1=x\right.\right)\leq \\
&
\leq 8 N \max_{x\in[0,1]^d}(\eta_{\sigma}(x)-\eta_{\sigma_0}(x))^2,
\end{align*}
where the last inequality follows from Lemma 1, \citet{castro1}. 
Also, note that we have $\max_{x\in[0,1]^d}$ in our bounds rather than the average over $x$ that would appear in the passive learning framework.\\ 
It remains to choose $q, m$ in appropriate way:
set 
$q\simeq\lfloor C_1 N^{\frac{1}{2\beta+d-\beta\gamma}}\rfloor$ and 
$m=\lfloor C_2 q^{d-\beta\gamma}\rfloor$ where $C_1, \ C_2$ are such that 
$q^d\geq m\geq 1$ and $32C_{L,\beta}^2 N q^{-2\beta}<\frac{m}{64}$ which is possible for $N$ big enough. 
In particular, $mq^{-d}=O(q^{-\beta\gamma})$. 
Together with the bound (\ref{KL}), this gives
$$
\frac{1}{M}\sum_{\sigma\in \m H'}{\rm KL}(P_{\sigma}\|P_{\sigma^0})\leq 32 C_u^2 N q^{-2\beta} < \frac{m}{8^2}=\frac18\log|\m H'|,
$$
so that conditions of Theorem \ref{tsyb2.5} are satisfied.
Setting 
$$
f_{\sigma}(x):=\sign \eta_{\sigma}(x),
$$
 we finally have $\forall \sigma_1\ne\sigma_2\in \m H'$  
$$
d(f_{\sigma_1},f_{\sigma_2}):=\Pi(\sign \eta_{\sigma_1}(x)\ne\sign \eta_{\sigma_2}(x))\geq \frac{m}{8q^d}\geq C_4 N^{-\frac{\beta\gamma}{2\beta+d-\beta\gamma}},
$$
where the lower bound just follows by construction of our hypotheses. 
Since under the low noise assumption 
$R_P(\hat f_n)-R^*\geq c\Pi(\hat f_n\ne \sign\eta)^{\frac{1+\gamma}{\gamma}}$(see (\ref{noise})), 
we conclude that
\begin{align*}
&\inf_{\hat f_N}\sup_{P\in \m P_U^*(\beta,\gamma)} \Pr\left(R_{P}(\hat f_n)-R^*\geq C_4 N^{-\frac{\beta(1+\gamma)}{2\beta+d-\beta\gamma}}\right)\geq \\
&\geq \inf_{\hat f_N}\sup_{P\in \m P_U^*(\beta,\gamma)} 
\Pr\left(\Pi(\hat f_n(x)\ne\sign \eta_{P}(x))\geq \frac{C_4}{2}N^{-\frac{\beta\gamma}{2\beta+d-\beta\gamma}}\right)\geq \tau>0.
\end{align*}
\end{proof}
%########################################
%########################################
\subsection{Upper bounds for the excess risk}
\label{upper}
%########################################
Below, we present a new active learning algorithm which is computationally tractable, adaptive with respect to 
$\beta,\gamma$(in a certain range of these parameters) and can be applied in the nonparametric setting. 
We show that the classifier constructed by the algorithm attains the rates of Theorem \ref{lower_bound}, up to polylogarithmic factor, 
if $0<\beta\leq 1$ and $ \beta\gamma\leq d$
(the last condition covers  the most interesting case when the regression function hits or crosses the decision boundary in the interior of the support of $\Pi$; 
for detailed statement about the connection between the behavior of the regression function near the decision boundary with parameters $\beta, \ \gamma$, see Proposition 3.4 in \citet{tsyb2}). 
%However, the upper bounds hold over a smaller class $P_U^*(\beta,\gamma)$. 
The problem of adaptation to higher order of smoothness ($\beta>1$) is still awaiting its complete solution; 
we address these questions below in our final remarks.
\\ 
%The answer is positive if the regression function possesses some regularity properties and satisfies the low noise assumption.
For the purpose of this section, the regularity assumption reads as follows: 
there exists $0<\beta\leq 1$ such that $\forall x_1,x_2\in[0,1]^d$ 
\begin{align}\label{holder1}
&
|\eta(x_1)-\eta(x_2)|\leq B_1\|x_1-x_2\|_{\infty}^{\beta}
\end{align}
Since we want to be able to construct non-asymptotic confidence bands, some estimates on the size of constants in (\ref{holder1}) and {\it assumption \ref{holder2}} are needed. Below, we will additionally assume that
\begin{align*}
&
B_1\leq \log N\\
&
B_2\geq \log^{-1}N, 
\end{align*}
where $N$ is the label budget. This can be replaced by any known bounds on $B_1,B_2$.\\
Let  $A\in \sigma(\m F_m)$ with $A_{\Pi}:=A\cap \supp(\Pi)\ne \emptyset$.
Define 
$$
\hat \Pi_A(dx):=\Pi(dx|x\in A_{\Pi})
$$
and 
$d_m:=\dim \m F_m|_{A_{\Pi}}$.
Next, we introduce a simple estimator of the regression function on the set $A_{\Pi}$. 
Given the resolution level $m$ and an iid sample $(X_i,Y_i), \ i\leq N$ with  $X_i\sim \hat \Pi_A$, let 
\begin{equation}
\label{est}
\hat \eta_{m,A}(x):=
\sum_{i:R_i\cap A_{\Pi}\ne \emptyset}
\frac{\sum_{j=1}^N Y_j \m I_{R_i}(X_j)}{N\cdot\hat\Pi_A(R_i)}\m I_{R_i}(x)
\end{equation}
Since we assumed that the marginal $\Pi$ is known, the estimator is well-defined. 
The following proposition provides the information about concentration of $\hat \eta_m$ around its mean: 
\begin{proposition}\label{supnorm}
For all $t>0$,
\begin{align*}
\Pr\Bigg(\max_{x\in A_{\Pi}}\lvert\hat\eta_{m,A}(x)-& \bar\eta_m(x)\rvert\ \geq t\sqrt{\frac{2^{dm}\Pi(A)}{u_1 N}}\Bigg)\leq \\
&
\leq 
2d_m\exp\left(\frac{-t^2}{2(1+\frac t3\sqrt{2^{dm}\Pi(A)/u_1 N})}\right),
\end{align*}
%where ${\m F_m}|_{\supp(\Pi)}$ is the restriction of $\m F_m$ onto support of $\Pi$.
\end{proposition}
\begin{proof}
This is a straightforward application of the Bernstein's inequality to the random variables 
$$
S_N^i:=\sum_{j=1}^N Y_j\m I_{R_i}(X_j), \ i\in\left\{i:R_i\cap A_{\Pi}\ne \emptyset\right\},
$$
and the union bound: indeed, note that 
$\mb E(Y \m I_{R_i}(X_j))^2=\hat\Pi_A(R_i)$, 
so that 
$$
\Pr\left(\left|S_N^i-N\int_{R_i}\eta d\hat\Pi_A\right|\geq tN\hat\Pi_A(R_i)\right)
\leq 2\exp\left(-\frac{N\hat\Pi_A(R_i)t^2}{2+2t/3}\right),
$$  
and the rest follows by simple algebra using that 
$\hat\Pi_A(R_i)\geq \frac{u_1}{2^{dm}\Pi(A)}$ 
by the $(u_1,u_2)$-regularity of $\Pi$.
\end{proof}
Given a sequence of hypotheses classes $\m G_m, \ m\geq 1$, define the index set 
\begin{align}
\label{index_set}
&
\m J(N):=\left\{m\in \mb N: \ 1\leq  \dim \m G_m\leq \frac{N}{\log^2 N}\right\}
\end{align}
 - the set of possible ``resolution levels'' of an estimator based on $N$ classified observations(an upper bound corresponds to the fact that we want the estimator to be consistent). When talking about model selection procedures below, we will implicitly assume that the model index is chosen from the corresponding set $\m J$. The role of $\m G_m$ will be played by $\m F_m|_A$ for appropriately chosen set $A$.
We are now ready to present the active learning algorithm followed by its detailed analysis(see Table 1).
\\
%\begin{comment}
%Everywhere below, $R_{j,m}=\left[\frac{j-1}{2^m},\frac{j}{2^m}\right)$.
\begin{table}[t]
\centering
\label{alg1}
{\footnotesize
\begin{tabular}{|l|}
\hline
 $\mbox{\bf Algorithm 1a}$ \\
\hline
 $\mbox{{\bf input } label budget } N; \ \text{confidence }\alpha$;\\
 $\hat m_0=0, \ \mathcal{\hat F}_0:=\mathcal{F}_{\hat m_0}, \ \hat\eta_0\equiv  0$; \\ 
 $LB:=N$; 
\hspace{2cm} // {\it label budget} \\
 $N_0:=2^{\lfloor\log_2 \sqrt{N}\rfloor}$;\\
 $s^{(k)}(m,N,\alpha):=s(m,N,\alpha):=m(\log N+\log\frac{1}{\alpha})$;\\
 $k:=0$;\\
 \bf{while} $LB\geq 0$ {\bf do} \\
 $k:=k+1$; \\
 $N_k:=2N_{k-1}$;\\
 $\hat A_k:=\left\{x\in [0,1]^d: \ \exists f_1,f_2\in \hat{\mathcal{F}}_{k-1}, \sign(f_1(x))\ne\sign(f_2(x))\right\};$  \\
 %$\hat A_k:=\bigcap\left\{A\in \sigma(\m F_{\hat m_{k-1}}):\ A\supseteq {\rm Act}\right\}$ \quad // {\it active set} \\
 $\mbox{\bf if } \hat A_k\cap \supp(\Pi)=\emptyset \mbox{ \bf or } LB<\lfloor N_k\cdot \Pi(\hat A_k) \rfloor\mbox{  \bf then } $ \\
 $\qquad\qquad\mbox{ {\bf break; output }} \hat g:=\sign \hat \eta_{k-1}$ \\
 $\qquad\qquad \mbox{ \bf else}$ \\
 $\mbox{{\bf for }} i=1\ldots \lfloor N_{k}\cdot \Pi(\hat A_k)\rfloor$\\
 $\mbox{{\bf sample i.i.d} } \left(X_{i}^{(k)},Y_i^{(k)}\right) \mbox{{\bf with }} X_i^{(k)}
\sim\hat\Pi_k:=\Pi(dx|x\in\hat A_k);$   
 \\
 $\mbox{{\bf end for}}$;\\ 
$LB:=LB-\lfloor N_{k}\cdot \Pi(\hat A_k)\rfloor$; \\
 $\hat P_k:=\frac{1}{\lfloor N_{k}\cdot \Pi(\hat A_k)\rfloor}
\sum\limits_{i}\delta_{X^{(k)}_{i},Y^{(k)}_{i}}$  \quad // {\it ''active'' empirical measure}
\\
 $\hat m_k:=\argmin_{m\geq \hat m_{k-1}}\left[\inf_{f\in\m F_m} \hat P_k(Y-f(X))^2+K_1\frac{2^{dm}\Pi(\hat A_k)+s(m-\hat m_{k-1},N,\alpha)}{\lfloor N_{k}\cdot \Pi(\hat A_k)\rfloor}\right]$ 
  \\
%$\hat m_k:=\min\left\{m: 2^{m}\geq 2^{\hat j_k}\log N\right\};$  \hspace{2.15 cm} // {\it undersmoothing} \\
%\lfloor K\cdot\hat j_k \log_2\log N\rfloor ;$ 
%$ n_k:=\min\left\{n\in M: \ n\geq {n}_{k-1}: \ \bar{U}_n\leq\frac 12 \delta_{k+1}\right\};$ \\
 %$\hat f_k:=\argmin_{f\in \m F_{\hat m_k}}\hat P_k(Y-f(X))^2$;\\
 $\hat \eta_k:=\hat \eta_{\hat m_k,\hat A_{k}}$ \qquad // {\it see (\ref{est})}
 \\
 $\delta_{k}:=\tilde D\cdot \log^2\frac{N}{\alpha}\sqrt{\frac{2^{d\hat m_k}}{N_k}}$; 
%\qquad // {\it $\tilde D:=2(K+C)$, $K,C$ from (\ref{Z4})}
% $\delta_{k}:=\tilde D\cdot s(\hat m_k-\hat m_{k-1},N,\alpha)\sqrt{\frac{2^{d\hat m_k}}{N_k}}$; 
 \\
% $\hat{\m F}_k:=\left\{f\in \m F_{\hat m_k}: \ f|_{{\hat A}_k^c}\in\hat{\m F}_{k-1}|_{\hat{A}_k^c}, \ f|_{\hat A_k}\in\mathcal{F}_{\infty,\hat A_k}(\hat \eta_k;\delta_{k})\right\}$;\\
 $\hat{\m F}_k:=\left\{f\in \m F_{\hat m_k}: \ f|_{\hat A_k}\in\mathcal{F}_{\infty,\hat A_k}(\hat \eta_k;\delta_{k}), \ 
 \ f|_{[0,1]^d\setminus {\hat A}_k}\equiv\hat\eta_{k-1}|_{[0,1]^d\setminus {\hat A}_k}\right\}$;\\
 $\mbox{\bf{end; }}$ \\
 %$\mbox{\bf{output }} \hat f\in \hat{\m F}_k$ \\
\hline
\end{tabular} 
}
\caption{Active Learning Algorithm}
\end{table}
%\end{comment}
\par
{\bf Remark} 
Note that on every iteration,
{\bf Algorithm 1a} uses the whole sample to select the resolution level $\hat m_k$ and to build the estimator $\hat \eta_k$. 
While being suitable for practical implementation, this is not convenient for theoretical analysis. 
We will prove the upper bounds for a slighly modified version: 
namely, on every iteration $k$ labeled data 
is divided into two subsamples $S_{k,1}$ and $S_{k,2}$ of approximately equal size,
$|S_{k,1}|\simeq |S_{k,2}|\simeq\left\lfloor \frac 12 N_{k}\cdot \Pi(\hat A_k)\right\rfloor$.
%(so that $|S_{k,1}|\simeq |S_{k,2}|$). 
Then $S_{1,k}$ is used to select the resolution level $\hat m_k$ 
and $S_{k,2}$ - to construct $\hat \eta_k$.
We will call this modified version {\bf Algorithm 1b}.
\par   
As a first step towards the analysis of {\bf Algorithm 1b}, 
let us prove the useful fact about the general model selection scheme. 
Given an iid sample $(X_i,Y_i), \ i\leq N$, set $s_m=m(s+\log\log_2 N), \ m\geq1$ and 
\begin{align}\label{res1}
&
\hat m:=\hat m(s)=\argmin_{\substack{m \in \m J(N)}}
\left[
\inf_{f\in\m F_m} P_N(Y-f(X))^2+K_1\frac{ 2^{dm}+ s_m }{N}
\right]
\\
&
\bar m:=\min\left\{
m\geq 1: \ \inf_{f\in \m F_m}\mb E(f(X)-\eta(X))^2\leq K_2 \frac{ 2^{dm}}{N}
\right\}
\end{align}
\begin{theorem}\label{opt_selection}
There exist an absolute constant $K_1$ big enough such that, with probability 
$\geq 1-e^{-s}$,
$$
\hat m\leq \bar m
$$
\end{theorem}
\begin{proof}
See Appendix \ref{res_lvl}.
\end{proof}
Straightforward application of this result immediately yields the following: 
\begin{corollary}\label{resol_bound}
Suppose $\eta(x)\in \Sigma(\beta,L,[0,1]^d)$. 
Then, with probability $\geq 1-e^{-s}$,
$$
2^{\hat m}\leq C_1\cdot N^{\frac{1}{2\beta+d}} 
$$
\end{corollary}
\begin{proof}
By definition of $\bar m$, we have
\begin{align*}
\bar m& \leq 
1+\max\left\{
m : \ \inf_{f\in \m F_m} \mb E(f(X)-\eta(X))^2 > K_2\frac{2^{dm}}{N}
\right\}\leq \\
&
\leq 
1+\max\left\{
m:\ L^2 2^{-2\beta m}> K_2\frac{2^{dm}}{N}
\right\},
\end{align*}
and the claim follows.
\end{proof}
With this bound in hand, we are ready to formulate and prove the main result of this section:
\begin{theorem}\label{main}
Suppose that $P\in \m P_U^*(\beta,\gamma)$ with $B_1\leq \log N, \ B_2\geq \log^{-1}N$ and $\beta\gamma\leq d$. 
Then, with probability 
$\geq 1-3\alpha$,
the classifier $\hat g$ returned by {\bf Algorithm 1b} with label budget $N$ satisfies 
%\left(\frac{2\beta+d}{2\beta+d-\beta\gamma}\right)^{\frac{\beta(1+\gamma)}{2\beta+d-\beta\gamma}}
$$
R_{P}(\hat g)-R^*\leq {\rm Const}\cdot 
N^{-\frac{\beta(1+\gamma)}{2\beta+d-\beta\gamma}}\log^p\frac{N}{\alpha},
$$
where $p\leq\frac{2\beta\gamma(1+\gamma)}{2\beta+d-\beta\gamma}$ and $B_1, \ B_2$ are the constants from (\ref{holder1}) and {\it assumption} \ref{holder2}.
\end{theorem}
{\bf Remarks }
\begin{enumerate}
\item Note that when $\beta\gamma>\frac d 3$,  
$N^{-\frac{\beta(1+\gamma)}{2\beta+d-\beta\gamma}}$ is a {\it fast rate}, i.e., faster than $N^{-\frac{1}{2}}$; 
at the same time, the passive learning rate $N^{-\frac{\beta(1+\gamma)}{2\beta+d}}$ is guaranteed to be fast only when $\beta\gamma>\frac d 2$, see \citet{tsyb2}. \\
\item For 
$\hat\alpha\simeq N^{-\frac{\beta(1+\gamma)}{2\beta+d-\beta\gamma}}$  {\bf Algorithm 1b} returns a classifier 
$\hat g_{\hat \alpha}$ that satisfies
$$
\mb E R_{P}(\hat g_{\hat \alpha})-R^*\leq {\rm Const}\cdot 
N^{-\frac{\beta(1+\gamma)}{2\beta+d-\beta\gamma}}\log^p N.
$$
This is a direct corollary of Theorem \ref{main} and the inequality \\
$$
\mb E|Z|\leq t+\|Z\|_{\infty}\Pr(|Z|\geq t)
$$
\end{enumerate}
\begin{proof}
Our main goal is to construct high probability bounds for the size of the active sets defined by {\bf Algorithm 1b}. 
In turn,
these bounds depend on the size of the confidence bands for $\eta(x)$, 
and the previous result(Theorem \ref{opt_selection}) is used to obtain the required estimates. 
Suppose $L$ is the number of steps performed by the algorithm before termination; clearly, $L\leq N$.\\
Let $N_k^{\act}:=\lfloor N_k\cdot \Pi(\hat A_k)\rfloor$ 
be the number of labels requested on $k$-th step of the algorithm: this choice guarantees that the ''density'' of labeled examples doubles on every step. \\
Claim: the following bound for the size of the active set holds uniformly
for all $2\leq k\leq L$ with probability at least
$
1-2\alpha
$:
\begin{align}
\label{Z3}
\Pi(\hat A_{k})&\leq 
\, C N_k^{-\frac{\beta\gamma}{2\beta+d}}\left(\log \frac{N}{\alpha}\right)^{2\gamma}
\end{align}
It is not hard to finish the proof assuming (\ref{Z3}) is true: indeed, it implies 
 that the number of labels requested on step $k$ satisfies
$$
N_k^{\act}=\lfloor N_k\Pi(\hat A_k)\rfloor\leq C\cdot N_k^{\frac{2\beta+d-\beta\gamma}{2\beta+d}}
\left(\log \frac{N}{\alpha}\right)^{2\gamma}
$$
with probability $\geq 1-2\alpha$.
Since $\sum\limits_k N_k^{\act}\leq N $,
one easily deduces that on the last iteration $L$ we have
%\frac{2\beta+d-\beta\gamma}{2\beta+d}
\begin{equation}\label{Z7}
N_L\geq c\left(\frac{N}{\log^{2\gamma} (N/\alpha)}\right)^{\frac{2\beta+d}{2\beta+d-\beta\gamma}}
\end{equation}
To obtain the risk bound of the theorem from here, we apply inequality (\ref{sup}) 
\footnote{alternatively, inequality (\ref{square}) can be used but results in a slightly inferior logarithmic factor.}
from proposition \ref{risk_bound}:
\begin{equation}\label{final}
R_{P}(\hat g)-R^*\leq 
D_1\|(\hat\eta_L-\eta)\cdot\m I\left\{\sign \hat\eta_L\ne \sign \eta\right\}\|_{\infty}^{1+\gamma}
%R(\hat g)-R_*\leq
%\left(\|\hat \eta_L-\eta\|^2_{L_2(\hat\Pi_L)}\Pi(\hat A_L)\right)^{\frac{1+\gamma}{2+\gamma}}
\end{equation}  
It remains to estimate $\|\hat \eta_L-\eta\|_{\infty,\hat A_L}$: we will show below while proving (\ref{Z3}) that 
$$
\|\hat \eta_L-\eta\|_{\infty,\hat A_L}\leq 
C\cdot N_{L}^{-\frac{\beta}{2\beta+d}}\log^2 \frac{N}{\alpha}
$$
Together with (\ref{Z7}) and (\ref{final}), it implies the final result.
\par
To finish the proof, it remains to establish (\ref{Z3}). 
Recall that $\bar \eta_k$ stands for the $L_2(\Pi)$ - projection of $\eta$ onto $\m F_{\hat m_k}$.
An important role in the argument is played by the bound on the $L_2(\hat\Pi_k)$ - norm of 
the ``bias'' $(\bar \eta_{k}-\eta)$: 
together with {\it assumption} \ref{holder2}, it allows to estimate 
$\|\bar \eta_{k}-\eta\|_{\infty,\hat A_k}$.   
The required bound follows from the following oracle inequality:  
there exists an event $\m B$ of probability $\geq 1-\alpha$ 
such that on this event for every $1\leq k\leq L$
\begin{align} \label{oracle}
\|\bar{\eta}_{k}-\eta\|_{L_2(\hat\Pi_k)}^2\leq 
\inf_{m\geq \hat m_{k-1}}
\Bigg[
\inf_{f\in \m F_m}&\|f-\eta\|_{L_2(\hat\Pi_k)}^2+ \\
& \nonumber
+K_1\frac{2^{dm}\Pi(\hat A_k) +(m-\hat m_{k-1})\log(N/\alpha)}{N_k \Pi(\hat A_k)}
\Bigg]
\end{align}
It general form, this inequality is given by Theorem 6.1 in \citet{kolt6} and provides the estimate for 
$\|\hat\eta_k-\eta\|_{L_2(\hat\Pi_k)}$, so it automatically implies the weaker bound for the bias term only. 
To deduce (\ref{oracle}), we use the mentioned general inequality $L$ times(once for every iteration) and the union bound. 
The quantity $2^{dm}\Pi(\hat A_k)$ in (\ref{oracle}) plays the role of the dimension, which is justified below. 
Let $k\geq 1$ be fixed.
For  $m\geq \hat m_{k-1}$, 
consider  hypothesis classes 
$$
\m F_m|_{\hat A_k}:=\left\{f\m I_{\hat A_k}, \ f\in \m F_m\right\}
$$
An obvious but important fact is that for $P\in \m P_U(\beta,\gamma)$, 
the dimension of $\m F_m|_{\hat A_k}$ is bounded by $u_1^{-1}\cdot 2^m \Pi(\hat A_k)$: indeed, 
$$
\Pi(\hat A_k)=\sum_{j:R_j\cap \hat A_k\ne\emptyset} 
\Pi(R_j)\geq u_1 2^{-dm}\cdot \#\left\{j:R_j\cap \hat A_k\ne\emptyset\right\},
$$
hence
\begin{equation}\label{dimension}
{\rm dim }\, \m F_m|_{\hat A_k}=\#\left\{j:R_j\cap \hat A_k\ne\emptyset\right\}\leq u_1^{-1}\cdot 2^m \Pi(\hat A_k).
\end{equation}
Theorem \ref{opt_selection} applies conditionally on 
$\left\{X_i^{(j)}\right\}_{i=1}^{N_j}, \ j\leq k-1$
with sample of size
$N_k^{\act}$
%(viewed as a sample from $\P|_{\hat A_k}$) 
and $s=\log(N/\alpha)$:
to apply the theorem, note that, by definition of $\hat A_k$, it is independent of $X_i^{(k)}, \ i=1\ldots N_k^{\act}$.
Arguing as in Corollary \ref{resol_bound} and using (\ref{dimension}), 
we conclude that the following inequality holds with probability 
$\geq 1-\frac{\alpha}{N}$ for every fixed $k$: 
\begin{align}
& \label{Z1}
2^{\hat m_k}\leq C\cdot N_k^{\frac{1}{2\beta+d}}.
\end{align}
Let $\m E_{1}$ be an event of probability $\geq 1-\alpha$ such that on this event bound 
(\ref{Z1}) holds for every step $k$, $k\leq L$ and let $\m E_2$ be an event of probability $\geq 1-\alpha$ on which inequalities (\ref{oracle}) are satisfied.
Suppose that event $\m E_1\cap \m E_2$ occurs and let $k_0$ be a fixed arbitrary integer $2\leq k_0\leq L+1$. 
It is enough to assume that $\hat A_{k_0-1}$ is nonempty(otherwise, the bound trivially holds), so that it contains at least one cube with sidelength $2^{-\hat m_{k_0-2}}$ and
\begin{equation}\label{Z5}
\Pi(\hat A_{k_0-1})\geq 
u_1 2^{-d\hat m_{k_0-1}}\geq c N_{k_0}^{-\frac{d}{2\beta+d}}
\end{equation}
Consider inequality (\ref{oracle}) with $k=k_0-1$ and $2^m\simeq N_{k_0-1}^{\frac{1}{2\beta+d}}$. 
By (\ref{Z5}), we have 
\begin{equation}
\label{Z8}
 \|\bar\eta_{k_0-1}-\eta\|^2_{L_2(\hat\Pi_{k_0-1})}\leq 
C N_{k_0-1}^{-\frac{2\beta}{2\beta+d}} \log^2 \frac{N}{\alpha}
\end{equation}
For convenience and brevity, denote $\Omega:=\supp(\Pi)$.
Now {\it assumption \ref{holder2}} comes into play: it implies, together with (\ref{Z8}) that
\begin{align}
 \label{supnorm1}
C N_{k_0-1}^{-\frac{\beta}{2\beta+d}}\log\frac{N}{\alpha}
&\geq
\|\bar \eta_{k_0-1}-\eta\|_{L_2(\hat\Pi_{k_0-1})}\geq 
B_2\|\bar \eta_{k_0-1}-\eta\|_{\infty,\Omega\cap\hat A_{k_0-1}}
\end{align}

To bound 
$$
\|\hat\eta_{k_0-1}(x)-\bar\eta_{k_0-1}(x)\|_{\infty,\Omega\cap\hat A_{k_0-1}}
$$
we apply Proposition \ref{supnorm}. 
Recall that $\hat m_{k_0-1}$ depends only on the subsample $S_{k_0-1,1}$ 
but not on $S_{k_0-1,2}$. Let 
$$
\m T_{k}:=\left\{\left\{X_i^{(j)},Y_i^{(j)}\right\}_{i=1}^{N_j^{\act}}, \ j\leq k-1; \  S_{k,1}\right\}
$$
be the random vector that defines $\hat A_{k}$ and resolution level $\hat m_{k}$. Note that
$
\mb E(\hat\eta_{k_0-1}(x)|\m T_{k_0-1})=\bar\eta_{\hat m_{k_0-1}}(x) \quad \forall x \ \text{a.s.}
$ \\
Proposition \ref{supnorm} thus implies
\begin{align*}
 \nonumber
 \Pr\Bigg(\max_{x\in\Omega\cap\hat A_{k_0-1}}\lvert\hat\eta_{k_0-1}(x)-\bar\eta_{\hat m_{k_0-1}}(x)\rvert\
 &\geq K t\sqrt{\frac{2^{d\hat m_{k_0-1}}}{N_{k_0-1}}} \,
 \Bigg|\, \m T_{k_0-1}\Bigg)\leq \\
&
\leq
N \exp\left(\frac{-t^2}{2(1+\frac t3 C_3)}\right).
\end{align*}
Choosing $t=c\log(N/\alpha)$ and taking expectation, the inequality(now unconditional) becomes 
\begin{equation}\label{supnorm2}
\Pr\left(\max_{x\in\Omega\cap\hat A_{k_0-1}}\lvert\hat\eta_{\hat m_{k_0-1}}(x)-\bar\eta_{\hat m_{k_0-1}}(x)\rvert\
 \leq K \sqrt{\frac{2^{d\hat m_{k_0-1}}\log^2(N/\alpha)}{N_{k_0-1}}}\right)
\geq 1-\alpha
\end{equation} 
%Set the constant $\tilde D$ in the definition of $\delta_k$ in Algorithm 1 to be equal to $2K$ from (\ref{supnorm2}).\\
Let $\m E_3$ be the event on which (\ref{supnorm2}) holds true.
%With bounds (\ref{Z1}), (\ref{Z8}) in hand, we can bound $\|\hat \eta_{k_0-1}-\eta\|_{\infty,\hat A_{k_0-1}}$:
Combined, the estimates (\ref{Z1}),(\ref{supnorm1}) and (\ref{supnorm2}) imply that on $\m E_{1}\cap \m E_{2}\cap \m E_3$
\begin{align}\label{Z4}
\nonumber
\|\eta-\hat \eta_{k_0-1}\|_{\infty,\Omega\cap\hat A_{k_0-1}}&\leq 
\|\eta-\bar\eta_{k_0-1}\|_{\infty,\Omega\cap\hat A_{k_0-1}}+
\|\bar\eta_{k_0-1}-\hat \eta_{k_0-1}\|_{\infty, \Omega\cap\hat A_{k_0-1}}\\
& 
\leq
\frac{C}{B_2} N_{k_0-1}^{-\frac{\beta}{2\beta+d}}\log \frac{N}{\alpha}
+K\sqrt{\frac{2^{d\hat m_{k_0-1}}\log^2 (N/\alpha)}{N_{k_0-1}}}
\leq \\
& \nonumber
\leq (K+C)\cdot N_{k_0-1}^{-\frac{\beta}{2\beta+d}}\log^2 \frac{N}{\alpha}
\end{align} 
where we used the assumption $B_2\geq \log^{-1}N$.
Now the width of the confidence band is defined via 
\begin{equation}\label{width}
\delta_k:=
2(K+C)\cdot N_{k_0-1}^{-\frac{\beta}{2\beta+d}}\log^2 \frac{N}{\alpha}
\end{equation} 
(in particular, $\tilde D$ from {\bf Algorithm 1a} is equal to $2(K+C)$).
%for $N$ big enough(so that the bias term is dominated by the variance term).
With the bound (\ref{Z4}) available, it is straightforward to finish the proof of the claim. 
%Recalling the definition of $\delta_k$(the size of the confidence band) and using (\ref{Z4}), 
Indeed, by (\ref{width}) and the definition of the active set, the necessary condition for 
$x\in \Omega\cap\hat A_{k_0}$ is 
$$
|\eta(x)|\leq 3(K+C)\cdot N_{k_0-1}^{-\frac{\beta}{2\beta+d}}\log^2 \frac{N}{\alpha},
$$
 so that 
\begin{align*}
\Pi(\hat A_{k_0})=\Pi(\Omega\cap\hat A_{k_0})&
\leq \Pi\left(|\eta(x)|\leq 3(K+C)\cdot N_{k_0-1}^{-\frac{\beta}{2\beta+d}}\log^2 \frac{N}{\alpha}\right)\leq \\
&
\leq \tilde B N_{k_0-1}^{-\frac{\beta\gamma}{2\beta+d}}\log^{2\gamma} \frac{N}{\alpha}
\end{align*}
by the low noise assumption.
This completes the proof of the claim since $\Pr\left(\m E_1\cap \m E_2\cap \m E_3\right)\geq 1-3\alpha$. 
\end{proof}
We conclude this section by discussing running time of the active learning algorithm. 
Assume that the algorithm has access to the sampling subroutine that, given $A\subset [0,1]^d$ with $\Pi(A)>0$, generates i.i.d. $(X_i,Y_i)$ with  $X_i\sim \Pi(dx|x\in A)$.
\begin{proposition}
\label{runningtime}
The running time of {\bf Algorithm 1a(1b)} with label budget $N$ is 
$$
\m O(dN\log^2 N).
$$
\end{proposition}
{\bf Remark } In view of Theorem \ref{main}, the running time required to output a classifier $\hat g$ such that
$
R_{P}(\hat g)-R^*\leq \eps
$
with probability $\geq 1-\alpha$ is 
$$
\m O\left(\left(\frac{1}{\eps}\right)^{\frac{2\beta+d-\beta\gamma}{\beta(1+\gamma)}}{\rm poly}\left(\log\frac{1}{\eps\alpha}\right)\right).
$$
\begin{proof}
We will use the notations of Theorem \ref{main}.
Let $N_k^{\act}$ be the number of labels requested by the algorithm on step $k$. 
The resolution level $\hat m_k$ is always chosen such that $\hat A_k$ is partitioned into at most $N_k^{\act}$ dyadic cubes, see (\ref{index_set}).
This means that the estimator $\hat \eta_k$ takes at most $N_k^{\act}$ distinct values. 
The key observation is that for any $k$, the active set $\hat A_{k+1}$ is always represented as the union of a finite number(at most $N_{k}^{\act}$) of dyadic cubes: 
to determine if a cube $R_j\subset \hat A_{k+1}$, it is enough to take a point $x\in R_j$ and compare 
${\rm sign}(\hat\eta_{k}(x)-\delta_{k})$ with ${\rm sign}(\hat\eta_{k}(x)+\delta_{k})$:
$R_j\in \hat A_{k+1}$ only if the signs are different(so that the confidence band crosses zero level). This can be done in $\m O(N_k^{\act})$ steps.\\
Next, resolution level $\hat m_k$ can be found in $\m O(N_k^{\act}\log^2 N)$ steps: there are at most $\log_2 N_k^{\act}$ models  to consider; for each $m$, $\inf_{f\in \m F_m}\hat P_k(Y-f(X))^2$ is found explicitly and is achieved for the piecewise-constant  
$$
\hat f(x)=\frac{\sum_i Y_i^{(k)}\m I_{R_j}(X_i^{(k)})}{\sum_i \m I_{R_j}(X_i^{(k)})}, \ x \in R_j.
$$
Sorting of the data required for this computation is done in $\m O(dN_k^{\act}\log N)$ steps for each $m$, so the whole $k$-th iteration running time is $\m O(dN_k^{\act}\log^2 N)$. 
Since $\sum\limits_k N_k^{\act}\leq N$, the result follows.
\end{proof}

%################################%
%\subsection{Numerical results}
%################################%
\section{Conclusion and open problems}
\label{conclusion}
%\subsection{Discussion of the main results} 
We have shown that active learning can significantly improve the quality of a classifier over the passive algorithm for a large class of underlying distributions. 
Presented method achieves fast rates of convergence for the excess risk, moreover, it is adaptive(in the certain range of smoothness  and noise parameters) and involves minimization only with respect to quadratic loss(rather than the $0-1$ loss).
\\ 
The natural question related to our results is:
\begin{itemize}
% \item Can we improve the lower bounds to hold over the same class as the upper bounds?
 \item Can we implement adaptive smooth estimators in the learning algorithm to extend our results beyond the case $\beta\leq 1$?
\end{itemize}
%The answer to the first question is most likely positive and is related to the technical flaws in our proof. However, at this point we are not able to close the gap.\\
The answer to this second question is so far an open problem. 
Our conjecture is that the correct rate of convergence for the excess risk is
$
N^{-\frac{\beta(1+\gamma)}{2\beta+d-\gamma(\beta\wedge 1)}},
$
up to logarithmic factors,
which coincides with presented results for $\beta\leq 1$.
This rate can be derived from an argument similar to the proof of Theorem \ref{main} under the assumption that on every step $k$ one could construct an estimator $\hat\eta_k$ with
$$
\|\eta-\hat\eta_k\|_{\infty, \hat A_k}\lesssim N_k^{-\frac{\beta}{2\beta+d}}.
$$
At the same time, the active set associated to $\hat\eta_k$ should maintain some structure which is suitable for the iterative nature of the algorithm. 
Transforming these ideas into a rigorous proof is a goal of our future work.
\section*{Acknowledgements}
%###################################
I want to express my deepest gratitude to my Ph.D. advisor, Dr. Vladimir Koltchinskii, for his support and numerous 
helpful discussions.\\
I am grateful to the anonymous reviewers for carefully reading the manuscript. 
Their insightful and wise suggestions helped to improve the quality of presentation and results.\\
I would like to acknowledge support for this project
from the National Science Foundation (NSF Grants DMS-0906880 and CCF-0808863) and by the Algorithms and Randomness Center, Georgia Institute of Technology, through the ARC Fellowship.
%##########APPENDICES##############
\appendix
\section{Functions satisfying assumption \ref{holder2}}\label{examples}
In the propositions below, we will assume for simplicity that the marginal distribution $\Pi$ is absolutely continuous with respect to Lebesgue measure with density $p(x)$ such that 
\begin{align}
\label{density}
&
0<p_1\leq p(x)\leq p_2<\infty \text{ for all } x\in [0,1]^d
\end{align}
Given $t\in(0,1]$, define $A_{t}:=\left\{x:\ |\eta(x)| \leq t\right\}$.
%\subsection{Proof of Lemma \ref{g-v}}
\begin{proposition}
\label{holder2:example1}
Suppose $\eta$ is Lipschitz continuous with Lipschitz constant $S$. 
Assume also that for some $t_*>0$ we have 
\begin{enumerate}[(a)]
\item
$\Pi\left(A_{t_*/3}\right)>0$;
\item
$\eta$ is twice differentiable for all $x\in A_{t_*}$;
\item 
$
\inf_{x\in A_{t_*}}\|\nabla\eta(x)\|_1 \geq s>0;
$
\item
$
\sup_{x\in A_{t_*}}\|D^2\eta(x)\|\leq C<\infty
$
where $\|\cdot\|$ is the operator norm.
\end{enumerate}
Then $\eta$  satisfies {\it assumption} \ref{holder2}.
\end{proposition}

\begin{proof}
By intermediate value theorem, for any cube $R_i, \ 1\leq i\leq 2^{dm}$ there exists $x_0\in R_i$ such that $\bar\eta_m(x)=\eta(x_0), \ x\in R_i$. 
This implies
\begin{align*}
|\eta(x)-\bar\eta_m(x)|&=|\eta(x)-\eta(x_0)|=|\nabla\eta(\xi)\cdot(x-x_0)|\leq \\
&
\leq \|\nabla\eta(\xi)\|_1 \|x-x_0\|_{\infty}\leq S \cdot 2^{-m}
\end{align*}
On the other hand, if $R_i\subset A_{t_*}$ then
\begin{align}\label{lower2}
\nonumber
|\eta(x)-\bar\eta_m(x)|&=
|\eta(x)-\eta(x_0)|= \\
&
\nonumber
=|\nabla\eta(x_0)\cdot (x-x_0)+ 
\frac12 [D^2\eta(\xi)](x-x_0)\cdot (x-x_0)|\geq \\
& 
\geq |\nabla\eta(x_0)\cdot (x-x_0)|-\frac12\sup_{\xi}\|D^2\eta(\xi)\|\max_{x\in R_i}\|x-x_0\|_2^2\geq \\
&
\nonumber
\geq
 |\nabla\eta(x_0)\cdot (x-x_0)|-C_1 2^{-2m}
\end{align}
Note that a strictly positive continuous function 
$$
h(y,u)=\int\limits_{[0,1]^d}(u\cdot(x-y))^2 dx
$$ 
achieves its minimal value $h_* > 0$ on a compact set $[0,1]^d\times \left\{u\in \mb R^d: \|u\|_1=1\right\}$. 
This implies(using (\ref{lower2}) and the inequality $(a-b)^2\geq \frac{a^2}{2}-b^2)$
\begin{align*}
&
\Pi^{-1}(R_i)\int\limits_{R_i}(\eta(x)-\bar\eta_m(x))^2 p(x)dx\geq \\
&
\geq \frac12(p_2 2^{dm})^{-1}\int\limits_{R_i}(\nabla \eta(x_0)\cdot(x-x_0))^2 p_1 dx -
C_1^2 2^{-4m}\geq \\
&
\geq \frac12\frac{p_1}{p_2}\|\nabla\eta(x_0)\|_1^2 2^{-2m}\cdot h_*-C_1^2 2^{-4m}
\geq c_2 2^{-2m} \quad \text{  for } m\geq m_0.
\end{align*}

Now take a set $A\in \sigma(\m F_m), \ m\geq m_0$ from {\it assumption} \ref{holder2}. 
There are 2 possibilities: either $A\subset A_{t_*}$ or $A\supset A_{t_*/3}$.
In the first case the computation above implies 
\begin{align*}
 \int\limits_{[0,1]^d}\left(\eta-\bar\eta_m\right)^2 \Pi(dx|x\in A)&
 \geq c_2 2^{-2m} = \frac{c_2}{S^2} S^2 2^{-2m}\geq \\
 &
  \geq  \frac{c_2}{S^2}  \|\eta-\bar\eta_m\|_{\infty, A}^2
\end{align*}
If the second case occurs, note that, since $\left\{x: 0<|\eta(x)|<\frac{t_*}{3}\right\}$ has nonempty interior, 
it must contain a dyadic cube $R_*$ with edge length $2^{-m_*}$. 
Then for any $m\geq \max(m_0,m_*)$
\begin{align*}
&
 \int\limits_{[0,1]^d}\left(\eta-\bar\eta_m\right)^2 \Pi(dx|x\in A)\geq \\
 &
\geq  \Pi^{-1}(A)\int\limits_{R_*}\left(\eta-\bar\eta_m\right)^2 \Pi(dx)\geq \frac{c_2}{4}2^{-2m}\Pi(R_*)\geq \\
& \geq  
\frac{c_2}{S^2}\Pi(R_*)  \|\eta-\bar\eta_m\|_{\infty, A}^2
\end{align*}
and the claim follows. 
\end{proof}
The next proposition describes conditions which allow functions to have vanishing gradient on decision boundary but requires convexity and regular behaviour of the gradient. \\
Everywhere below, $\nabla\eta$ denotes the subgradient of a convex function $\eta$.\\
For $0<t_1<t_2$, define 
$G(t_1,t_2):=
\frac{\sup\limits_{x\in A_{t_2}\setminus A_{t_1}} \|\nabla \eta(x)\|_1}{\inf\limits_{x\in A_{t_2}\setminus A_{t_1}} \|\nabla \eta(x)\|_1}$. 
In case when $\nabla\eta(x)$ is not unique, we choose a representative that makes $G(t_1,t_2)$ as small as possible.
\begin{proposition}
\label{holder2:example2}
Suppose $\eta(x)$ is Lipschitz continuous with Lipschitz constant $S$. 
Moreover, assume that there exists $t_*>0$ and $q:(0,\infty)\mapsto (0,\infty)$ such that $A_{t_*}\subset~(0,1)^d$ and
\begin{enumerate}[(a)]
\item $b_1t^{\gamma}\leq\Pi(A_{t})\leq b_2 t^{\gamma} \ \forall t<t_*$;
\item For all $0<t_1<t_2\leq t_*, $ 
$G(t_1,t_2)\leq q\left(\frac{t_2}{t_1}\right)$;
\item Restriction of $\eta$ to any convex subset of $A_{t_*}$ is convex.
\end{enumerate}
Then $\eta$ satisfies {\it assumption} \ref{holder2}.\\
{\bf Remark} The statement remains valid if we replace $\eta$ by $|\eta|$ in (c).
\end{proposition}
\begin{proof}
Assume that for some $t\leq t_*$ and $k>0$ 
$$
R\subset A_{t}\setminus A_{t/k}
$$ 
is a dyadic cube with edge length $2^{-m}$ and let $x_0$ be such that 
$\bar\eta_m(x)=\eta(x_0), \ x\in R$. 
Note that $\eta$ is convex on $R$ due to (c).
Using the subgradient inequality $\eta(x)-\eta(x_0)\geq\nabla\eta(x_0)\cdot(x-x_0)$, we obtain
\begin{align}\label{subgrad}
\nonumber
&
\int\limits_{R}(\eta(x)-\eta(x_0))^2 d\Pi(x)\geq
\int\limits_{R}(\eta(x)-\eta(x_0))^2 \m I\left\{\nabla\eta(x_0)\cdot(x-x_0)\geq 0\right\}d\Pi(x) \\
&
\geq \int\limits_{R}\left(\nabla\eta(x_0)\cdot(x-x_0)\right)^2 \m I\left\{\nabla\eta(x_0)\cdot(x-x_0)\geq 0\right\} d\Pi(x)
\end{align}
The next step is to show that under our assumptions $x_0$ can be chosen such that
\begin{align}\label{ball}
&
\dist_{\infty}(x_0,\partial R)\geq \nu 2^{-m}
\end{align}
where $\nu=\nu(k)$ is independent of $m$.
In this case any part of $R$ cut by a hyperplane through $x_0$ contains half of a ball $B(x_0,r_0)$ of radius 
$r_0=\nu(k) 2^{-m}$ and the last integral in (\ref{subgrad}) can be further bounded below to get
\begin{align}
\label{subgrad2}
\nonumber
\int\limits_{R}(\eta(x)-\eta(x_0))^2 d\Pi(x)
&
\geq
\frac12\int\limits_{B(x_0,r_0)}\left(\nabla\eta(x_0)\cdot(x-x_0)\right)^2 p_1 dx\geq  \\
&
\geq c(k)\|\nabla\eta(x_0)\|_1^2 2^{-2m}2^{-dm}
\end{align}
It remains to show (\ref{ball}). 
Assume that for all $y$ such that $\eta(y)=\eta(x_0)$ we have 
$$
\dist_{\infty}(y,\partial R)\leq \delta 2^{-m}
$$
for some $\delta>0$. 
This implies that the boundary of the convex set 
$$
\left\{x\in R:\eta(x)\leq \eta(x_0)\right\}
$$ 
is contained in 
$R_{\delta}:=\left\{x\in R: \dist_{\infty}(x,\partial R)\leq \delta2^{-m}\right\}$.
There are two possibilities: either $\left\{x\in R:\eta(x)\leq \eta(x_0)\right\}\supseteq R\setminus R_{\delta}$ or
$\left\{x\in R:\eta(x)\leq \eta(x_0)\right\}\subset R_{\delta}$.\\
We consider the first case only(the proof in the second case is similar). 
First, note that by (b) for all $x\in R_{\delta}$ $\|\nabla\eta(x)\|_1\leq q(k)\|\nabla\eta(x_0)\|_1$ and
\begin{align}
\label{A1}
\nonumber
\eta(x)&\leq \eta(x_0)+\|\nabla\eta(x)\|_1\delta 2^{-m}\leq \\
&
\leq \eta(x_0)+q(k)\|\nabla\eta(x_0)\|_1 \delta2^{-m}
\end{align}
Let $x_c$ be the center of the cube $R$ and 
 $u$ - the unit vector in direction $\nabla \eta(x_c)$.
Observe that
\begin{align*}
\eta(x_c+(1-3\delta)2^{-m} u)-\eta(x_c)&
\geq\nabla\eta(x_c)\cdot (1-3\delta)2^{-m }u=\\
&
=(1-3\delta)2^{-m}\|\nabla\eta(x_c)\|_2
\end{align*}
On the other hand, $x_c+(1-3\delta)2^{-m} u\in R\setminus R_{\delta}$ and
$$
\eta(x_c+(1-3\delta)2^{-m} u)\leq \eta(x_0),
$$ 
hence
$
\eta(x_c)\leq \eta(x_0)-c(1-3\delta)2^{-m}\|\nabla\eta(x_c)\|_1
$.
Consequently, for all 
$$
x\in B(x_c,\delta):=\left\{x:\|x-x_c\|_{\infty}\leq \frac12 c2^{-m}(1-3\delta)\right\}
$$ 
we have 
\begin{align}\label{A2}
\nonumber
\eta(x)&\leq \eta(x_c)+\|\nabla\eta(x_c)\|_1\|x-x_c\|_{\infty}\leq \\
&
\leq
\eta(x_0)-\frac12 c2^{-m}(1-3\delta)\|\nabla\eta(x_c)\|_1
\end{align}
Finally, recall that $\eta(x_0)$ is the average value of $\eta$ on $R$. 
Together with (\ref{A1}),(\ref{A2}) this gives
\begin{align*}
\Pi(R)\eta(x_0)&=\int\limits_R \eta(x)d\Pi
=\int\limits_{R_\delta}\eta(x)d\Pi+\int\limits_{R\setminus R_\delta}\eta(x)d\Pi\leq \\
&
\leq (\eta(x_0)+q(k)\|\nabla\eta(x_0)\|_1\delta2^{-m}) \Pi(R_\delta)+\\
&
+(\eta(x_0)- c_2 2^{-m}(1-3\delta)\|\nabla\eta(x_0)\|_1)\Pi\left(B(x_c,\delta)\right)+\\
&
+\eta(x_0)\Pi(R\setminus(R_{\delta}\cup B(x_c,\delta)))=\\
&
=\Pi(R)\eta(x_0)+q(k)\|\nabla\eta(x_0)\|_1\delta2^{-m}\Pi(R_\delta)-\\
&- c_2 2^{-m}(1-3\delta)\|\nabla\eta(x_0)\|_1\Pi\left(B(x_c,\delta)\right)
\end{align*}
Since 
$\Pi(R_\delta)\leq p_2 2^{-dm}$ and 
$\Pi(B(x_c,\delta))\geq c_3 2^{-dm}(1-3\delta)^d$, the inequality above implies
$$
c_4 q(k)\delta\geq (1-3\delta)^{d+1}
$$
which is impossible for small $\delta$(e.g., for $\delta< \frac{c}{q(k)(3d+4)}$).\\
Let $A$ be a set from condition \ref{holder2}. 
If $A\supseteq A_{t_*/3}$, then there exists a dyadic cube $R_*$ with edge length $2^{-m_*}$ such that 
$R_*\subset A_{t_*/3}\setminus A_{t_*/k}$ for some $k>0$, and the claim follows from (\ref{subgrad2}) as in proposition \ref{holder2:example1}.\\
Assume now that $A_{t}\subset A\subset A_{3t}$ and $3t\leq t_*$. 
Condition (a) of the proposition implies that for any $\eps>0$ we can choose $k(\eps)>0$ large enough so that 
\begin{align}
\label{measure1}
&
\Pi(A\setminus A_{t/k})\geq\Pi(A)-b_2(t/k)^{\gamma}\geq
\Pi(A)-\frac{b_2}{b_1}k^{-\gamma}\Pi(A_t)\geq (1-\eps)\Pi(A)
\end{align}
This means that for any partition of $A$ into dyadic cubes $R_i$ with edge length $2^{-m}$ at least half of them satisfy
\begin{align}
\label{index}
&
\Pi(R_i\setminus A_{t/k})\geq (1-c\eps)\Pi(R_i)
\end{align}
Let $\m I$ be the index set of cardinality $|\m I |\geq c \Pi(A) 2^{dm-1}$ such that (\ref{index}) is true for $i\in \m I$. 
Since $R_i\cap A_{t/k}$ is convex, there exists
\footnote{If, on the contrary, every sub-cube with edge length $2^{-(m+z)}$ contains a point from $A_{t/k}$, then $A_{t/k}$ must contain the convex hull of these points which would contradict (\ref{measure1}) for large $z$.}
$z=z(\eps)\in \mb N$ such that for any such cube $R_i$ there exists a dyadic sub-cube with edge length $2^{-(m+z)}$ entirely contained in $R_i\setminus A_{t/k}$:
$$
T_i\subset R_i\setminus A_{t/k}\subset A_{3t}\setminus A_{t/k}.
$$
It follows that 
$\Pi\big(\bigcup\limits_i T_i\big)\geq \tilde c(\eps)\Pi(A)$.
Recall that condition (b) implies 
$$
\frac{\sup\limits_{x\in \cup_i T_i} \|\nabla \eta(x)\|_1}{\inf\limits_{x\in \cup_i T_i} \|\nabla \eta(x)\|_1}\leq q(3k)
$$
%Finally, condition (d) implies that for $t$ small enough the boundary of the set $A_{3t}$ does not intersect with the boundary of the unit cube, which means that 
Finally, $\sup\limits_{x\in A_{3t}} \|\nabla\eta(x)\|_2$ is attained at the boundary point, that is for some $x_*: |\eta(x_*)|=3t$, and by (b) 
$$
\sup_{x\in A_{3t}}\|\nabla\eta(x)\|_1\leq
\sqrt{d}\|\nabla\eta(x_*)\|_1\leq
q(3k)\sqrt{d} \inf\limits_{x\in A_{3t}\setminus A_{t/k}}\|\nabla\eta(x)\|_1.
$$
Application of (\ref{subgrad2}) to every cube $T_i$ gives
\begin{align*}
\sum_{i\in \m I}& \int\limits_{T_i}(\eta(x)-\bar\eta_{m+z}(x))^2  d\Pi(x)\geq 
c_1(k)\Pi(A)|\m I |\inf\limits_{x\in A_{3t}\setminus A_{t/k}}\|\nabla\eta(x)\|_1^2 2^{-2m}2^{-dm}\geq \\
&
\geq c_2(k)\Pi(A)\sup\limits_{x\in A_{3t}} \|\nabla\eta(x)\|_1^2 2^{-2m}\geq c_3(k)\Pi(A)\|\eta-\bar\eta(m)\|_{\infty,A}^2
\end{align*}
concluding the proof.
\end{proof}

\section{Proof of Theorem \ref{opt_selection}}\label{res_lvl}
The main ideas of this proof, which significantly simplifies and clarifies initial author's version, are due to V. Koltchinskii.
For conveniece and brevity, let us introduce additional notations. 
Recall that 
$$
s_m=m(s+\log\log_2 N)
$$ 
Let 
\begin{align*}
&
\tau_N(m,s):=K_1\frac{2^{dm}+s_m}{N} \\
&
\pi_N(m,s):=K_2\frac{2^{dm}+s+\log\log_2 N}{N}
\end{align*}
By $\m E_P(\m F,f)$ (or $\m E_{P_N}(\m F,f)$) we denote the excess risk of 
$f\in \m F$ with respect to the true (or empirical) measure:
\begin{align*}
& 
\m E_P(\m F,f):=P(y-f(x))^2-\inf_{g\in \m F}P(y-g(x))^2 \\
&
\m E_{P_N}(\m F,f):=P_N(y-f(x))^2-\inf_{g\in \m F}P_N(y-g(x))^2
\end{align*} 
It follows from Theorem 4.2 in \citet{kolt6} and the union bound that there exists an event $\m B$ of probability $\geq 1-e^{-s}$ such that on this event the following holds for all $m$ such that $dm\leq \log N$:
\begin{align}\label{event}
& \nonumber
\m E_P(\m F_m,\hat{f}_{\hat m})\leq \pi_{N}(m,s) \\
\forall \ f\in \m F_m, 
\quad & \m E_P(\m F_m,f)\leq 2(\m E_{P_N}(\m F_m,f)\vee \pi_{N}(m,s))
\\ \nonumber
\forall \ f\in \m F_m, 
\quad& \m E_{P_N}(\m F_m,f)\leq \frac32(\m E_P(\m F_m,f)\vee \pi_{N}(m,s)).
\end{align}
We will show that on $\m B$, $\left\{\hat m\leq \bar m\right\}$ holds. 
Indeed, assume that, on the contrary, $\hat m>\bar m$; by definition of $\hat m$, we have
$$
P_N(Y-\hat f_{\hat m})^2+\tau_N(\hat m,s)\leq P_N(Y-\hat f_{\bar m})^2+\tau_N(\bar m,s),
$$
which implies 
$$
\m E_{P_N}(\m F_{\hat m},\hat f_{\bar m})\geq 
\tau_N(\hat m,s)-\tau_N(\bar m,s)>3\pi_N(\hat m,s)
$$
for $K_1$ big enough. By (\ref{event}), 
$$
\m E_{P_N}(\m F_{\hat m},\hat f_{\bar m})=
\inf\limits_{f\in \m F_{\bar m}}\m E_{P_N}(\m F_{\hat m},f)\leq 
\frac32\left(\inf_{f\in \m F_{\bar m}}\m E_P(\m F_{\hat m},f)\vee \pi_{N}(\hat m,s)\right),
$$
and combination the two inequalities above yields
\begin{equation}\label{A}
\inf_{f\in \m F_{\bar m}}\m E_P(\m F_{\hat m},f)>
\pi_N(\hat m,s)
\end{equation}
Since for any $m$ $\m E_P(\m F_{m},f)\leq \mb E(f(X)-\eta(X))^2$,  the definition of $\bar m$ and (\ref{A}) imply that
$$
\pi_N(\bar m,s)\geq \inf_{f\in \m F_{\bar m}}\mb E(f(X)-\eta(X))^2>\pi_N(\hat m,s),
$$ 
contradicting our assumption, hence proving the claim.
\nocite{gaiffas1}
\nocite{kolt5}
%\nocite{gine1}
\nocite{tsyb1}
\nocite{wellner1}
\nocite{hanneke1}
\nocite{hanneke2}
\nocite{balcan1}
\nocite{tsyb2}
\nocite{gaiffas1}
\nocite{nickl1}
%\nocite{castro1}
%\nocite{gine2}
%\bibliographystyle{apalike}
%\bibliographystyle{astron}
%\bibliographystyle{alpha}
%\bibliographystyle{harvard}
\bibliographystyle{plainnat}
\bibliography{bibliography}

\end{document}